\documentclass[reqno]{amsart}
\usepackage{pdfsync}

\usepackage{amsmath}%
\usepackage{amsthm}%
\usepackage{amsfonts}
\usepackage{amssymb}%
\usepackage{graphicx}
\usepackage{color}
\usepackage{subfigure}
\newtheorem{thm}{Theorem}[section]
\theoremstyle{plain}

\newtheorem{cor}[thm]{Corollary}

\newtheorem{prop}[thm]{Proposition}
\theoremstyle{remark}

\newtheorem{remark}[thm]{Remark}

\newenvironment{subproof}[1][\proofname]
{%
  \begin{proof}[#1]%
}
{%
  \end{proof}%
}

\newcommand{\R}{\mathbb{R}}
\newcommand{\E}{\mathbb{E}}

\newcommand{\T}{\mathcal{T}}
\newcommand{\cE}{\mathcal{E}}

\numberwithin{equation}{section}

\begin{document}
\title[Quasilinear Parabolic Eqns and Continuous Max Reg]
{On Quasilinear Parabolic Equations and Continuous Maximal Regularity}

\author[J. LeCrone]{Jeremy LeCrone}
\address{Department of Mathematics \& Computer Science \\ University of Richmond\\ Richmond, VA USA}
\email{jlecrone@richmond.edu}%
\urladdr{http://math.richmond.edu/faculty/jlecrone/}

\author[G. Simonett]{Gieri Simonett}
\address{Department of Mathematics \\ Vanderbilt University \\ Nashville, TN USA}
\email{gieri.simonett@vanderbilt.edu}%
\urladdr{http://www.vanderbilt.edu/math/people/simonett}

\thanks{This work was supported by a grant from the Simons Foundation (\#426729, Gieri Simonett).}

\subjclass[2010]{Primary 35K90, 35K59, 35B30, 35B35; Secondary 53C44, 35K93} 

\keywords{quasilinear parabolic equations, continuous maximal regularity, generalized principle
of linearized stability, critical spaces, well--posedness, surface diffusion flow, stability of cylinders,
stability of spheres}%

\begin{abstract}
We consider a class of abstract quasilinear parabolic problems with lower--order terms 
exhibiting a prescribed singular structure. 
We prove well--posedness and Lipschitz continuity of associated semiflows.
Moreover, we investigate global existence of solutions and we extend 
the generalized principle of linearized stability to settings with initial values in critical spaces.
These general results are applied to the surface diffusion flow in various settings.

\end{abstract}

\maketitle

\section{Introduction}

In this paper, we consider abstract quasilinear parabolic evolution equations given by
\begin{equation}\label{QuasiEqn}
	\begin{cases}
		\dot{u} + A(u)u = F_1(u) + F_2(u), &\text{for $t >0$},\\
		u(0) = x,
	\end{cases}
\end{equation}
for which we extend previous well--posedness and global existence results in the setting of continuous maximal regularity.
As a particular feature, we admit nonlinearities $F_2$ with a prescribed singular structure.

In more detail, we assume that $(E_1,E_0)$ is a pair of Banach spaces so that $E_1$ is densely embedded in $E_0$,
and we seek solutions of \eqref{QuasiEqn} in time weighted spaces 
\[
\E_{1,\mu}(J):= BC_{1 - \mu}^1(J,E_0) \cap BC_{1 - \mu}(J,E_1),
\]
where $\mu\in (0,1)$ and $J=[0,T]$, see Section~\ref{sec:WellPosedness} for a precise definition.
Let $V_\mu\subset E_\mu$ be an open subset of the continuous interpolation space 
$E_{\mu} := (E_0, E_1)_{\mu, \infty}^{0}$,
and set $\E_{0,\mu}(J):=BC_{1-\mu}(J,E_0)$.
Then we assume that the functions $(A,F_1,F_2)$ satisfy the following conditions.
\begin{itemize}
	\item[{\bf (H1)}] Local Lipschitz continuity of $(A, F_1)$: 
		\begin{equation}\label{Regularity}
			(A, F_1) \in C^{1-} \big(V_{\mu}, \mathcal{M}_{\mu}(E_1, E_0) \times E_0 \big).
		\end{equation}
	Here $\mathcal{M}_{\mu}(E_1, E_0)$ denotes the set of  bounded linear operators
	$B \in \mathcal{L}(E_1,E_0)$ for which $(\E_{1,\mu}(J), \E_{0,\mu}(J))$
	is a pair of maximal regularity for some $T>0$ (and hence all $T>0$). Therefore,
	$B \in \mathcal{M}_{\mu}(E_1,E_0)$ if and only if for every 
	$(f,x) \in \E_{0,\mu}(J) \times E_\mu$ there exists a unique function $u \in \E_{1,\mu}(J)$
	such that 
	\[
		\dot{u}(t) + Bu(t) = f(t) \quad \text{for } t \in \dot J , \quad \text{and } u(0) = x.
	\]

	\item[{\bf (H2)}] Structural regularity of $F_2$: \\
	        There exists a number $\beta\in (\mu,1)$  such that 
	        $F_2: V_\mu\cap E_\beta\to E_0$.
	        Moreover, there are numbers  $\beta_j \in [\mu, \beta]$, $\rho_j \ge 0$ and 
	        $m \in \mathbb{N}$ with 
	     \begin{equation}\label{Criticality}
			\frac{\rho_j(\beta - \mu) + (\beta_j  - \mu)}{1 - \mu} \le 1, \qquad 
			\text{for all $j = 1, \ldots, m$},
		\end{equation}
		so that for each $x_0 \in V_{\mu}$ and $R > 0$
		there is a constant $C_R = C_R(x_0) > 0$ for which the estimate
		\begin{equation}\label{Structural}
			|F_2(x_1) - F_2(x_2)|_{E_0} \le 
				C_R \sum_{j = 1}^{m} \big( 1 + |x_1|_{E_{\beta}}^{\rho_j} + 
					|x_2|_{E_{\beta}}^{\rho_j} \big) |x_1 - x_2|_{E_{\beta_j}}
		\end{equation}
		holds for all $x_1, x_2 \in \bar{B}_{E_{\mu}}(x_0,R) \cap (V_\mu \cap E_{\beta}).$	
	
\end{itemize}
Following the convention introduced in Pr\"uss, Wilke \cite{PW17}, we call the index $j$ \textit{subcritical} 
if \eqref{Criticality} is a strict inequality, and  \textit{critical} in case equality holds in \eqref{Criticality}. 
As $\beta_j\le \beta<1$, any $j$ with $\rho_j=0$ is subcritical.
Furthermore, \eqref{Criticality} is equivalent to $\rho_j\beta+\beta_j- 1\le \rho_j\mu$.
Hence, the minimum value of $\mu$ is given by
\begin{equation*}
\mu_{\rm crit}:=\beta-\min_{\rho_j\neq 0} \left(\frac{1-\beta_j}{\rho_j}\right).
\end{equation*}
The number $\mu_{\rm crit}$ is called the \textit{critical weight} and $E_{\mu_{\rm crit}}$ a \textit{critical space}.

In the last decades, there has been an increasing interest in finding critical spaces for nonlinear parabolic partial differential equations. 
As a matter of fact, there is no generally accepted definition 
in the mathematical literature concerning the notion of critical spaces.
One possible definition may be based on the idea of a `largest space of initial data such that the given PDE is well-posed.'
Critical spaces are often introduced as  `scaling invariant spaces,' provided the underlying PDE enjoys a scaling invariance.
It has been shown in~\cite{PSW18} that the concept of \textit{critical weight} and \textit{critical space} introduced there 
(and also used in this paper) captures and unifies
the idea of largest space and scaling invariant space. In more detail, it has been shown in~\cite{PSW18} that
 $E_{\mu_{\rm crit}} $ is, in a generic sense, the largest space of initial data for which the given equation is well-posed, and 
that  $E_{\mu_{\rm crit}}$ is scaling invariant, provided the given equation admits a scaling.

Our approach for establishing well--posedness of \eqref{QuasiEqn} relies on the concept of 
\textit{continuous maximal regularity} in time-weighted spaces and extends previous 
results by Angenent~\cite{An90}, Clement and Simonett~\cite{CS01}, Lunardi~\cite{Lun95},
and Asai~\cite{Asai12}.
The results parallel those in \cite{LPW14, PSW18, PW17}, 
where well--posedness of \eqref{QuasiEqn} is studied by means of maximal $L_p$-regularity in time-weighted function spaces.

Leveraging the singular structure of $F_2,$ along with inequalities from interpolation theory 
and continuous maximal regularity, we prove local well--posedness of \eqref{QuasiEqn} 
via a fixed point argument.
Allowing for rough initial values in $E_\mu$ with $\mu\ge \mu_{\rm crit}$, 
we prove Lipschitz continuity of the 
associated semiflow on $V_\mu$ and derive conditions for global 
well--posedness and asymptotic behavior of solutions near normally stable equilibria. 
A key feature of our results is that dynamic properties of solutions are controlled in the 
topology of $E_\mu$, rather than requiring further control in stronger topologies 
of $E_\beta$ or $E_1$ as solutions regularize.

In particular, we prove that a priori bounds in the topology of $E_\mu$ yield global existence.
Moreover, we extend the generalized principle of linearized stability (c.f. \cite{PSZ09b,PSZ09a}),
proving that solutions with initial data that is $E_\mu$--close to a normally stable equilibrium 
will converge exponentially fast to a nearby equilibrium.

As a particular application of our abstract results, we consider the
surface diffusion flow, a  geometric evolution equation acting on orientable hypersurfaces. 
Given a fixed reference manifold $\Sigma \subset \R^n$, 
we consider the evolution of surfaces $\Gamma(t)$ defined in normal direction over 
$\Sigma$ via time--dependent height functions $h : \Sigma \times \R^+ \to \R.$
The governing equation for surface diffusion is then expressed as a fourth--order, 
parabolic evolution law acting on $h= h(t)$  
and we look for solutions in the setting of little--H\"older continuous functions; i.e. 
\[
	E_0 := bc^{\alpha}(\Sigma) \quad \text{and} \quad E_1 := bc^{4 + \alpha}(\Sigma),
		\quad \text{for some $\alpha \in (0,1).$}
\]

Considering the setting of $\Sigma = \mathcal{C}_r \subset \R^3,$ 
an infinite cylinder of radius $r >~0$,
we work with height functions $h(t) : \mathcal{C}_r \to (-r, \infty)$ which produce so--called
{\it axially--definable} surfaces $\Gamma(h),$ as in \cite{LS16}. 
We show that the resulting surface diffusion flow can be cast as a quasilinear 
parabolic evolution equation in the form \eqref{QuasiEqn} with
\[
	E_\mu := bc^{1 + \alpha}(\Sigma) \quad \text{and} \quad 
		E_\beta := bc^{3 + \alpha}(\Sigma),
\]
from whence we have $\mu = 1/4$ and $\beta = 3/4$ in this setting. 
Explicitly expressing the singular nonlinearity $F_2$, we employ interpolation theory estimates
to confirm the necessary singular structure {\bf (H2)} is satisfied on $V_\mu \cap E_\beta,$
where $V_\mu$ is an appropriately chosen open subset of $E_\mu$.
The appearance of several critical indices supports the idea that $bc^{1+\alpha}(\Sigma)$
is in fact a critical space for surface diffusion flow.

Applying our general results to initial data $h_0 \in bc^{1+\alpha}(\mathcal{C}_r)$ we extend 
well--posedness from \cite[Proposition~3.2]{LS16} to surfaces with only one H\"older continuous derivative.
Further, we extend \cite[Proposition 2.2, 2.3]{LS13} by restricting to functions
$h_0 \in bc^{1+\alpha}_{symm}(\mathcal{C}_r)$ exhibiting azimuthal symmetry around the 
cylinder $\mathcal{C}_r$. Further, enforcing periodicity of $h_0$ along the central axis of $\mathcal{C}_r$, 
 we show stability and instability of cylinders under
periodic perturbations with H\"older control on only first order derivatives. In particular, when 
$r > 1$, we show that $2\pi$--periodic $\Gamma(h_0)$ surfaces that are $bc^{1+\alpha}$--close
to $\mathcal{C}_r$ give rise to global solutions to surface diffusion flow converging to a
nearby cylinder exponentially fast. On the other hand, when $r < 1$, we show that 
there exist $2\pi$--periodic perturbations
which are arbitrarily close to $\mathcal{C}_r$ in $bc^{1+\alpha}$ for which solutions escape a
neighborhood of the cylinder.
We also direct the reader to \cite{BBW98}, for additional information 
concerning the surface diffusion flow for axisymmetric surfaces.

Taking $\Sigma$ to be an arbitrary compact, connected, immersed manifold, we demonstrate 
well--posedness of surface diffusion for initial data in $bc^{1 + \alpha}(\Sigma)$ that are
sufficiently close to the manifold $\Sigma$, an extension of \cite[Theorem 1.1]{EMS98}.
Further, in case $\Sigma \subset \R^n$ is a Euclidean sphere, we apply our generalized stability 
result to yield stability of the family of spheres 
under perturbations which require control on only first--order derivatives.
This result extends \cite[Theorem 1.2]{EMS98}, 
where initial values in  $ bc^{2+\alpha}(\Sigma)$ are considered.

Working also in the setting of surfaces parameterized over a sphere,
Escher and Mucha \cite{EM10} show
that small perturbations in the topology of Besov spaces $B_{p,2}^{5/2-4/p}(\Sigma)$ 
exist globally and converge exponentially fast to a sphere.
Although the topologies of these Besov spaces and our little--H\"older spaces are not 
easily comparable, we note that our stability results hold for any spacial dimension $n$,
while the regularity of perturbations in \cite{EM10} changes with $n$.
In particular, Escher and Mucha enforce the bound $p > \frac{2n+6}{3}$, which
they note only guarantees existence of lower regularity perturbations in 
$C^{1+\alpha}(\Sigma) \setminus C^{2}(\Sigma)$ when $n < 9$. 
(Notice that the authors in \cite{EM10} consider surfaces in $\R^{n+1}.)$

Regarding a different approach to stability of spheres, we refer to
\cite{MW11,W11,W12} where the lifespan of solutions, and convergence to equilibria, 
is controlled via $L_2$--estimates of the second fundamental form.
We observe that our assumptions on initial data allow for initial surfaces on which the 
second fundamental form may not
be defined, so that our results are not contained in \cite{MW11,W11,W12}.

As a final remark on surface diffusion flow, we mention that several authors have considered the flow
of surfaces with rough initial data when $\Gamma \subset \R^{n+1}$ is given as the graph 
of a function over a domain $\Omega \subseteq \R^n$, 
see \cite{Asai12,KL12,LPW14} for instance. In \cite{KL12}, Koch and Lamm 
prove global existence of solutions to surface diffusion flow with initial surfaces that are merely 
Lipschitz continuous, and they prove analytic dependence on initial data. However, Koch and Lamm 
work in the setting of entire graphs (i.e. $\Omega = \R^n$) 
and require a smallness condition on Lipschitz
norm which seems to make it difficult to translate their result to more general settings.
Working also in the setting of entire graphs, the conclusions of Asai in \cite{Asai12} are 
closest to our current results, as the author works in spaces of little--H\"older continuous functions. 
We refer to Remark~\ref{rem:BeforeMainProof} for a detailed account 
of the results in~\cite{Asai12}.
In \cite{LPW14}, the authors approach surface diffusion flow from the setting of $L_p$ maximal 
regularity on a bounded domain $\Omega$, producing well--posedness for initial data in Besov spaces 
$B_{qp}^{4\mu-4/p}(\Omega),$ for an appropriate choice of $\mu$, $p$, and $q$. 

We briefly outline the current paper. 
In Section~\ref{sec:WellPosedness}, we 
state and prove our main result, Theorem~\ref{thm:MainResult}.
We conclude Section~\ref{sec:WellPosedness} with an extension of well--posedness, giving
continuous dependence on initial data in stronger topologies $E_{\bar \mu}$.

In Section~\ref{sec:Stability}, we prove equivalence between measuring stability of equilibria
in the space $E_\mu$ and measuring stability in a smaller space $E_{\bar \mu}$,
$\bar \mu \in [\beta,1)$. We then prove the generalized principle of linearized stability 
for perturbations in $E_\mu$.

In Section~\ref{sec:SDFlow} we apply our results to various settings for surface diffusion flow.
Beginning with {\it axially--definable} surfaces parameterized
over an infinite cylinder, we conclude well--posedness of surface diffusion flow for 
general perturbations in $bc^{1+\alpha}$. Then we enforce periodicity in the general
setting to establish stability / instability of cylinders under perturbations in $bc_{per}^{1+\alpha}$
(with radius $r$ above / below the threshold $r = 1$), before producing similar results 
in the setting of axisymmetric surfaces. We end Section~\ref{sec:SDFlow} with the 
setting of surfaces defined over an arbitrary compact reference manifold $\Sigma$,
and establish well--posedness and stability of spheres with initial data in $bc^{1+\alpha}.$

\section{Well--Posedness of \eqref{QuasiEqn}}\label{sec:WellPosedness}
In this section, we formulate and prove our main result concerning solvability of \eqref{QuasiEqn}.
Moreover, we formulate and prove conditions for global existence of solutions.
We start with the definition and elementary properties of \textit{time--weighted continuous spaces}
(see \cite[Section 2]{CS01} for more details).

Let $E$ be an arbitrary Banach space
and define the spaces of \textit{time--weighted} continuous functions
\[
	BC_{1 - \mu}(J, E) := \Big\{ v \in C(\dot{J}, E) : 
			[ t \mapsto t^{1 - \mu}v(t)] \in BC(\dot{J},E), \ \lim_{t \to 0^+} |t^{1 - \mu} v(t)|_{E}  = 0\Big\}
\]
and
\[
	BC_{1-\mu}^1(J,E) := \{ v \in C^1(\dot{J}, E) : v, \dot{v} \in BC_{1-\mu}(J, E) \},
\]
where $J:=[0,T],$ $\dot J:=(0,T],$ and $\mu \in (0,1)$. We also set
\[
	BC_{0}(J,E) := C(J, E), \quad \text{and} \quad BC_0^1(J,E) := C^1(J,E).
\]

Given Banach spaces $E_0$ and $E_1$ so that $E_1$ is densely embedded in $E_0$, we define
\begin{equation}\label{MaxRegSpaces}
\begin{split}
	\E_{1,\mu}(J) &:= BC_{1 - \mu}^1(J,E_0) \cap BC_{1 - \mu}(J,E_1), \qquad \text{and}\\
	\E_{0,\mu}(J) &:= BC_{1 - \mu}(J,E_0),
\end{split}
\end{equation}
which are themselves Banach spaces when equipped with the norms
\[
\begin{split}
	\| v \|_{\E_{1,\mu}(J)} &:= \sup_{t \in J} t^{1 - \mu} 
		\big( |\dot{v}(t)|_{E_0} + |v(t)|_{E_1} \big), \qquad \text{and}\\
	\| v \|_{\E_{0,\mu}(J)} &:= \sup_{t \in J} t^{1 - \mu} |v(t)|_{E_0}
\end{split}
\]
respectively. Further, we note that the trace operator $\gamma : \E_{1,\mu}(J) \to E_0$ is 
well--defined and, assuming $\mathcal{M}_\mu(E_1,E_0) \ne \emptyset$ (as we do 
throughout), the trace space
$\gamma \E_{1,\mu}(J)$ coincides with the continuous interpolation space
$E_\mu := (E_0, E_1)_{\mu, \infty}^0$.

\begin{remark} \label{InterpConstants}
{\bf (a)}
The important inequality \eqref{Criticality} should be viewed in relation to applications
of interpolation we will encounter frequently in the article. 
In particular, if we set 
\[
	\alpha := \frac{\beta - \mu}{1 - \mu} \qquad \text{and} \qquad
	\alpha_j := \frac{\beta_j - \mu}{1 - \mu}
\]
then $E_{\beta} = (E_\mu, E_1)_{\alpha,\infty}^0$ and 
$E_{\beta_j} = (E_\mu, E_1)_{\alpha_j,\infty}^0,$ for $j = 1, \ldots, m .$
Then, given $x, y \in E_1$ and $t > 0$, it follows that
\begin{align}\label{CriticalInterp}
	t^{1 - \mu}|x|_{E_\beta}^{\rho_j}|y|_{E_{\beta_j}} 
		&\le c_j 
			t^{1 - \mu}\big(|x|_{E_\mu}^{\rho_j(1 - \alpha)}|x|_{E_1}^{\rho_j \alpha}\big)
			\big(|y|_{E_\mu}^{(1 - \alpha_j)}|y|_{E_1}^{\alpha_j}\big)\\
		&\le  C_0 
			t^{(1-\mu)(1 - \rho_j \alpha - \alpha_j)}
			|x|_{E_\mu}^{\rho_j(1 - \alpha)}|t^{1-\mu} x|_{E_1}^{\rho_j \alpha}
			|y|_{E_\mu}^{(1 - \alpha_j)}|t^{1 - \mu} y|_{E_1}^{\alpha_j}. \notag
\end{align}
Here the constant $c_j = c_j(\alpha, \alpha_j)$ is the product of interpolation constants from
$E_\beta$ and $E_{\beta_j}$ (c.f. \cite[Proposition 2.2.1]{Am95}), while 
$C_0 = C_0(\alpha, \alpha_1, \ldots, \alpha_m)$ is an upper bound for the family
of all such constants, $j = 1, \ldots, m$.

\medskip
{\bf (b)}
In the proof of Theorem~\ref{thm:MainResult} below,
we address both subcritical and critical indices $j$.
The difference in approaches to these two cases can 
be viewed in context of \eqref{CriticalInterp}.
In particular, note that when $j$ is subcritical, the exponent 
$$(1-\mu)(1-\rho_j \alpha - \alpha_j)$$ 
on $t$ is strictly positive, since 
$\rho_j \alpha + \alpha_j$ is exactly the left--hand side of \eqref{Criticality}.
Meanwhile, when $j$ is critical we have a trivial exponent on 
$t$, but it must hold that $\rho_j > 0$ in this case. Thus, when
$j$ is critical we focus on the term $|t^{1-\mu} x|_{E_1}^{\rho_j \alpha}$ which
has a positive exponent (a property
not necessarily holding in the subcritical case).
\end{remark}

\goodbreak

\begin{thm}\label{thm:MainResult}
	Suppose $A, F_1$ and $F_2$ satisfy conditions {\bf (H1)--(H2)}.
\begin{itemize}
\vspace{1mm}\item[{\bf (a)}] {\bf (Local Solutions)}
	Given any $x_0 \in V_{\mu}$, there exist positive constants $\tau = \tau(x_0)$, 
	$\varepsilon = \varepsilon(x_0)$, and $\sigma = \sigma(x_0)$ such that
	\eqref{QuasiEqn} has a unique solution
	\[
		u(\cdot, x) \in  \E_{1,\mu}([0,\tau]):= 
		BC_{1 - \mu}^1([0,\tau], E_0) \cap BC_{1 - \mu}([0, \tau], E_1)
	\]  
	for all initial values $x \in \bar{B}_{E_{\mu}}(x_0, \varepsilon)$.
	Moreover, 
\begin{equation}\label{Lip-mu}
\qquad\qquad \| u(\cdot, x_1) - u(\cdot, x_2) \|_{\E_{1,\mu}([0,\tau])}  
	\le \sigma | x_1 - x_2 |_{E_\mu},
	\quad  x_1, x_2 \in \bar{B}_{E_{\mu}}(x_0, \varepsilon).
\end{equation}
\item[{\bf (b)}] {\bf (Maximal Solutions)}
Each solution with initial value $x_0\in V_\mu$ exists on a maximal interval 
$J(x_0) := [0,t^+) = [0,t^+(x_0))$ and enjoys the regularity
$$
	u(\cdot,x_0)\in C([0,t^+), E_\mu)\cap C((0,t^+), E_1).
$$
\item[{\bf (c)}] {\bf (Global Solutions)}
If the solution $u(\cdot, x_0)$ satisfies the conditions: 
\begin{itemize}
	\vspace{1mm}\item[(i)] $u(\cdot, x_0) \in UC(J(x_0), E_\mu)$ and 
	\vspace{1mm} \item[(ii)] there exists $\eta > 0$ so that ${\rm dist}_{E_\mu}(u(t,x_0), \partial V_\mu) > \eta$
		for all $t \in J(x_0)$,
\end{itemize}
\vspace{1mm}
then it holds that $t^+(x_0) = \infty$ and so $u(\cdot, x_0)$ is a global solution of \eqref{QuasiEqn}.
Moreover, if the embedding $E_1 \hookrightarrow E_0$ is compact, then condition (i) may be replaced
by the assumption:
\begin{itemize}
	\vspace{1mm}\item[(i.a)] the orbit $\{ u(t, x_0) : t \in [\tau, t^+(x_0))\}$ is bounded in $E_\delta$
		for some $\delta \in (\mu, 1]$ and some $\tau\in (0, t^+(u_0)).$
\end{itemize}
\end{itemize}
\end{thm}

Before proceeding to the proof of the theorem we add some remarks.
We first note that the embedding
$\E_{1,\mu}([0,\tau]) \hookrightarrow C([0,\tau],E_\mu),$
see \cite[Lemma 2.2(b)]{CS01},
immediately implies the following result. 

\begin{cor}\label{SemiflowRegularity}
Under the assumptions of Theorem~\ref{thm:MainResult}, there
exists a positive constant $c = c(x_0)$ so that 
\[
	\| u(\cdot, x_1) - u(\cdot,x_2) \|_{C([0,\tau],E_\mu)} \le 
		c|x_1 - x_2|_{E_\mu} \qquad 
		\text{for $x_1, x_2 \in \bar{B}_{E_\mu}(x_0, \varepsilon).$}
\]	
It thus follows that the map $[(t,x) \mapsto u(t,x)]$ defines
a locally Lipschitz continuous semiflow on $V_\mu.$
\end{cor}
\medskip

\begin{remark}\label{rem:BeforeMainProof}
{\bf (a)} We recall briefly that local Lipschitz continuity of a semiflow on $V_\mu$ means that
\[
	\mathcal{D} := \bigcup_{x \in V_\mu} [0,t^+(x)) \times \{x\}
\]
is an open set in $\R_+ \times V_\mu,$ the map $[(t,x) \mapsto u(t,x)]$ is continuous
on $\mathcal{D},$ and for all $(t_0, x_0) \in \mathcal{D}$ there exists a product neighborhood
$U \times V \subset \mathcal{D}$ and $c > 0$ so that 
\[
	|u(t,x) - u(t,y)|_{V_\mu} \le c|x-y|_{V_\mu} \qquad \text{for $(t,x), (t,y) \in U \times V$.}
\]

\medskip
{\bf (b)}
Local well--posedness of \eqref{QuasiEqn} was also considered by Asai \cite{Asai12} 
in the presence of a singular right--hand side $F: V_\mu \cap E_1 \to E_0$. 
In particular, the author assumes that $F$ satisfies 
\[
	|F(x_1) - F(x_2)|_{E_0} \le C_R 
		\big( 1 + |x_1|^p_{E_1} + |x_2|^p_{E_1} \big) |x_1 - x_2|_{E_\theta}
\]
for all $x_1, x_2 \in E_1 \cap B_{E_\mu}(x_0, R).$ 
Here the author has $p$ and 
$E_\theta := (E_0, E_1)_{\theta, \infty}^0$ appropriately chosen, with 
$\theta \in [\mu,1),$ so that $p + (\theta - \mu)/(1 - \mu)< 1$.
This setting is similar to our condition {\bf (H2)} if one allows $j =1$, $\beta = 1$,
$\rho_j = p$, and $\beta_j = \theta$, whereby it follows that Asai
only considers subcritical weights.
Further, we note that in \cite[Theorem 1.1]{Asai12} the author proves H\"older continuous
dependence on initial data in $V_\mu$, whereas we obtain Lipschitz continuity.
No additional geometric properties for solutions are established in \cite{Asai12}.
\end{remark}
\medskip

\begin{proof}[Proof of Theorem 2.2]
(a)
We follow the structure of related proofs in \cite{LPW14} and \cite{PW17}, where $L_p$--maximal 
regularity is assumed. We note that sub--critical and critical indices required distinct
proofs in \cite{LPW14} and \cite{PW17}, respectively, whereas both cases can be handled in the same
setting here.

Choose $x_0 \in V_\mu$ and fix $\varepsilon_0 > 0$ so that 
\[
	\bar{B}_{E_{\mu}}(x_0, \varepsilon_0) \subset V_{\mu}.
\]
Applying {\bf (H1)} and {\bf (H2)}, we obtain constants $L=L(\varepsilon_0) > 0$ and $C_{\varepsilon_0} > 0$ so that 
\begin{equation}\label{Structure1}
	\| (A,F_1)(x_1) - (A,F_1)(x_2) \|_{\mathcal{L}(E_1,E_0) \times E_0} \le L |x_1 - x_2|_{E_\mu},
\end{equation}
for $x_1, x_2 \in \bar{B}_{E_\mu}(x_0, \varepsilon_0),$ and
\begin{equation}\label{Structure2}
	|F_2(x_1) - F_2(x_2)|_{E_0} \le C_{\varepsilon_0} \sum_{j = 1}^m 
		\big( 1 + |x_1|_{E_{\beta}}^{\rho_j} + |x_2|_{E_{\beta}}^{\rho_j}\big) 
		|x_1 - x_2|_{E_{\beta_j}}
\end{equation}
for $x_1, x_2 \in \bar{B}_{E_{\mu}}(x_0, \varepsilon_0) \cap E_{\beta}$.

\medskip

It follows from {\bf (H1)} and \cite[Corollary 1]{LS11}
that $-A(x_0)$ generates a strongly continuous analytic semigroup on $E_0$ with domain $E_1$.
For each element $x \in V_\mu$, we define 
\[
	u_x^\star(t) := e^{-t A(x_0)}x
\]
which is in $\E_{1,\mu}(J_T)$, for any $T>0$ and solves
$\dot{u} + A(x_0) u = 0, \ u(0) = x$. \\

Furthermore, we fix positive constants $T_1>0,$ $M_1 \ge 1$, and $C_1>0$ so that, 
for all $J_T := [0,T] \subset J_{T_1}$ and $x_1, x_2 \in \bar{B}_{E_\mu}(x_0, \varepsilon_0)$, we have
\begin{equation}\label{StrongCont}
	\| u_{x_0}^\star - x_0 \|_{C(J_T, E_\mu)} 
	< \frac{\varepsilon_0}{3},
\end{equation}
\begin{equation}\label{TraceZero}
	\| u \|_{C(J_T, E_\mu)} \le M_1 \| u \|_{\E_{1,\mu}(J_T)}, \quad 
		\text{for $u \in \E_{1,\mu}(J_T)$ with $u(0) = 0,$}
\end{equation}
and
\begin{equation}\label{ICStars}
	M_1 \| u_{x_1}^\star - u_{x_2}^\star \|_{\E_{1,\mu}(J_T)} + 
		\| u_{x_1}^\star - u_{x_2}^\star \|_{C(J_T,E_\mu)} 
		\le C_1 |x_1 - x_2|_{E_\mu}.
\end{equation}
The previous inequalities are justified by strong continuity of the semigroup $e^{-tA(x_0)}$ in $E_\mu$,
\cite[Lemma 2.2(c)]{CS01}, and \cite[Equation (3.7)]{CS01}, respectively.

We will construct a contraction mapping 
on a closed subset of $\E_{1,\mu}(J_T)$ given by
\begin{equation}\label{WxJTr}
	W_x(J_T,r) := \{ v \in \E_{1,\mu}(J_T): v(0) = x, \text{ and } 
		\|v - u_{x_0}^\star\|_{\E_{1,\mu}(J_T)} \le r \},
\end{equation}
where $x \in E_\mu$, $J_T \subset J_{T_1}$, and $r > 0$.
The mapping we consider will be $\T_x$ which takes $v \in W_x(J_T,r)$
to the solution $w = \T_x(v)$ 
of the linear initial value problem\\
\begin{equation}\label{TauMap}
	\begin{cases}
		\dot{w} + A(x_0)w = (A(x_0) - A(v)) v + F_1(v) + F_2(v), & t\in \dot J_T,\\
		w(0) = x.
	\end{cases}
\end{equation}
Note that fixed points $v = \T_x(v)$ are solutions to the original problem 
\eqref{QuasiEqn} on $J_T$.
We proceed by first proving that $\T_x$ is well--defined (see Claims 1 and 2 below),
then we show that $\T_x$ is in fact a contraction mapping on $W_x(J_T,r)$
for $r, T$ and $x$ appropriately chosen (see Claims 3 and 4 below).

\medskip


\noindent\underline{\bf Claim 1:} For $r, T, \varepsilon$ chosen sufficiently small and positive, if 
$x \in \bar{\mathbb{B}}_{E_{\mu}}(x_0, \varepsilon)$ then 
$\| v - x_0 \|_{C(J_T,E_\mu)} \le \varepsilon_0$ for all $v \in W_x(J_T,r)$. 

\begin{subproof}[Proof of Claim 1] 
		For any $v \in W_x(J_T,r)$ note that $v(0) - u_x^\star(0) = 0$,
		thus \eqref{TraceZero} and the triangle inequality imply
		\[
			\|v - u_x^\star\|_{C(J_T,E_\mu)} \le 
				M_1 \|v - u_{x_0}^\star\|_{\E_{1,\mu}(J_T)} + 
				M_1 \|u_{x_0}^\star - u_x^\star\|_{\E_{1,\mu}(J_T)}.
		\]
		Applying  \eqref{StrongCont}, \eqref{ICStars}, 
		and \eqref{WxJTr}, we compute
		\begin{align}\label{Claim1Control}
			 \| v - x_0 \|_{C(J_T,E_\mu)} \notag 
				&\le \| v - u_x^\star \|_{C(J_T,E_\mu)} + 
					\| u_x^\star - u_{x_0}^\star \|_{C(J_T,E_\mu)} + 
					\| u_{x_0}^\star - x_0 \|_{C(J_T,E_\mu)} \notag\\
				&\le M_1 r + C_1 |x - x_0 |_{E_\mu} + 
					\| u_{x_0}^\star - x_0 \|_{C(J_T,E_\mu)} \\
				&\le M_1 r + C_1 \varepsilon + \varepsilon_0 /3. \notag
		\end{align}
		Claim 1 thus follows by restricting $r$
		and $\varepsilon$ appropriately so that the last line  in \eqref{Claim1Control}
		is bounded by $\varepsilon_0$.
\end{subproof}

\medskip

Henceforth, we assume $x$ is sufficiently close to $x_0$ (in $E_{\mu}$) and 
$r, T, \varepsilon$ are given appropriately small so that Claim 1 holds.
It follows that given any $v_1, v_2 \in W_x(J_T,r),$ the 
structural conditions \eqref{Structure1}--\eqref{Structure2} hold for 
$v_1(s), v_2(t)$, with $s,t \in \dot J_T$.\\


\noindent\underline{\bf Claim 2:} 
$(A(x_0)-A(v))v,\; F_1(v),\; F_2(v) \in \mathbb{E}_{0,\mu}(J_T)$
for each $v \in W_x(J_T,r)$.

\begin{subproof}[Proof of Claim 2]
Regarding regularity of $(A(x_0) - A(v))v$ and
$F_1(v)$, for $v \in W_x(J_T,r)$, 
note that continuity of each function 
into $E_0$ follows from \eqref{Structure1} and the fact that 
$v \in \E_{1,\mu}(J_T) \hookrightarrow C(J_T,E_\mu)$. 
Employing {\bf (H1)} and the bounds 
\eqref{Structure1} and \eqref{Claim1Control}, we compute
\begin{equation}\label{ABound}
\begin{aligned}
	t^{1-\mu}|(A(x_0) - &A(v(t))) v(t)|_{E_0} \\
		&\le L |x_0 - v(t)|_{E_{\mu}}  t^{1-\mu} |v(t)|_{E_1} \\
		&\le L \| v - x_0 \|_{C(J_T,E_\mu)} \| v \|_{\mathbb{E}_{1,\mu}(J_T)} \\
		&\le L\big(M_1 r +C_1\varepsilon +\|u^\star_{x_0}-x_0\|_{C(J_T,E_\mu)}\big)
			\big(r+\|u^\star_{x_0}\|_{\E_{1,\mu}(J_T)}\big)
\end{aligned}
\end{equation}
and
\begin{align}\label{F1Bound}
	t^{1-\mu} |F_1(v(t))|_{E_0} 
		&\le t^{1-\mu} |F_1(v(t)) - F_1(x_0)|_{E_0} + t^{1-\mu} | F_1(x_0) |_{E_0} \notag\\
		&\le t^{1-\mu} L | v(t) - x_0 |_{E_{\mu}} + t^{1-\mu} | F_1(x_0)|_{E_0}\\
		&\le T^{1-\mu}\big(L \varepsilon_0 + | F_1(x_0)|_{E_0}\big). \notag
\end{align}
From  \eqref{ABound} and \eqref{F1Bound}, we draw the following conclusions.
For each $v \in W_x(J_T,r)$, we see that $t^{1-\mu} |F_1(v(t))|_{E_0}$ and 
$t^{1-\mu}|(A(x_0) - A(v(t)))v(t)|_{E_0}$ are bounded on $\dot{J}_T$. 
Further, as $T \to 0^+$ notice that $\| v \|_{E_{1,\mu}(J_T)} \to 0$ 
and $T^{1-\mu} \to 0,$ from which we conclude 
\begin{equation}\label{F1AMapping}
	(A(x_0)-A(v))v, \;F_1(v) \in \E_{0,\mu}(J_T), \quad \text{for all $v \in W_x(J_T,r)$}. 
\end{equation}
As an additional observation, note \eqref{ABound} and \eqref{F1Bound} imply
$\| (A(x_0) - A(v))v \|_{\E_{0,\mu}(J_T)}$
and $\| F_1(v) \|_{\E_{0,\mu}(J_T)}$ are uniformly bounded on $W_x(J_T,r).$\\

Lastly, we consider the term $F_2(v)$, and we first observe that 
\begin{equation}
\label{F2-continuity}
	F_2(v)\in C(\dot J_T,E_0)\;\; \text{for each}\; \;  v\in W_x(J_T,r).
\end{equation} 
Indeed, let $ v\in W_x(J_T,r)$ be given.
Then $v\in C(\dot J_T,E_\beta)$,  and by Claim 1, we also know that 
$v(t)\in \bar B_{E_\mu}(x_0,\varepsilon_0)$ for all $t\in J_T$.
Hence we have for each $s,t\in \dot J_T$
\begin{equation*}
	|F_2(v(t))-F_2(v(s))|_{E_0}\le C_{\varepsilon_0}\sum_{j=1}^m 
		\big(1+ |v(t)|^{\rho_j}_{E_\beta} +  |v(s)|^{\rho_j}_{E_\beta}\big)|v(t)-v(s)|_{E_{\beta_j}}.
\end{equation*}
The assertion in \eqref{F2-continuity} now follows from the embedding 
$E_\beta\hookrightarrow E_{\beta_j}$
and the observation that $|v(t)|_{E_\beta}$ and $|v(s)|_{E_\beta}$ are 
bounded for values $s,t$ that are bounded away from $0$.
The latter statement means that for each $\eta\in \dot J_T$ there is a constant $C_\eta>0$ such that 
$|v(t)|_{E_\beta}, |v(s)|_{E_\beta}\le C_\eta$ for all $s,t\in  [\eta,T]$.

In order to show boundedness of $t^{1-\mu}|F_2(v(t))|_{E_0}$ we 
choose $y\in E_1\cap B_{E_\mu}(x_0,\varepsilon_0)$, which is feasible by 
the density of the embedding $E_1\hookrightarrow E_\mu$, 
and  write 
\[
	F_2(v(t))= (F_2(v(t))-F_2(y)) + F_2(y),\quad t\in \dot J_T.
\]
Clearly, $F_2(y)\in \E_{0,\mu}(J_T)$. 
To treat the term $ F_2(v(t))-F_2(y),$ we first observe that
\begin{equation}
\label{v-y-1}
	\begin{aligned}
	|v(t)-y|_{E_{\beta_j}} 
		& \le c_j  t^{-(1-\mu)\alpha_j} |v(t)-y|^{1-\alpha_j}_{E_\mu} 
			(t^{1-\mu}|v(t)-y|_{E_1})^{\alpha_j} \\
		&  \le c_j  t^{-(1-\mu)\alpha_j} |v(t)-y|^{1-\alpha_j}_{E_\mu} 
			\|v-y\|^{\alpha_j}_{\E_{1,\mu}(J_T)} 
		\le \tilde c t^{-(1-\mu)\alpha_j}
	\end{aligned}
\end{equation}
for each $t\in \dot J_T$. 
Next, employing \eqref{CriticalInterp} and \eqref{v-y-1}, we have
\begin{equation*}
\begin{aligned}
	t^{1-\mu} &|F_2(v(t))-F_2(y)|_{E_0}\\
	&\le C_{\varepsilon_0} t^{1-\mu} \sum_{j=1}^m 
		(1+ |y|^{\rho_j}_{E_\beta} + |v(t)|^{\rho_j}_{E_\beta})|v(t)-y|_{E_{\beta_j}} \\
	&\le \tilde C \bigg[ \sum_{j=1}^m t^{(1-\mu)(1-\alpha_j)} (1+|y|^{\rho_j}_{E_\beta})
		+ \sum_{j \text{ critical}} 
		\|v\|_{C(J_T,E_\mu)}^{\rho_j(1-\alpha)}
		\|v\|_{\E_{1,\mu}(J_T)}^{\rho_j \alpha} \\
	&\qquad + \sum_{j \text{ subcritical}} 
		t^{(1-\mu)(1-\rho_j\alpha - \alpha_j)} 
		\|v\|_{C(J_T,E_\mu)}^{\rho_j(1-\alpha)}
		\|v\|_{\E_{1,\mu}(J_T)}^{\rho_j \alpha} 
	\bigg].
\end{aligned}
\end{equation*}
Since $\| v \|_{\E_{1,\mu}(J_T)} \to 0$ as $T \to 0^+$,
we see that $t^{1-\mu}F_2(v(t))$ is bounded in $E_0$
		and $\|F_2(v)\|_{\E_{0,\mu}(J_T)}$ converges to zero as $T \to 0^+.$ 
		Claim 2 is thus proved.
\end{subproof}

\medskip

With $T,r$ and $x$ chosen as above, we have now shown that the 
right hand side of equation \eqref{TauMap} is in $\E_{0,\mu}(J_T)$ for all $v \in W_x(J_T,r).$
Thus, since $A(x_0) \in \mathcal{M}_\mu(E_1,E_0)$, 
we conclude that $\T_x: W_x(J_T,r) \to \E_{1,\mu}(J_T)$ is well--defined. 
Now, to conclude the proof of the theorem, we must show that $\T_x$ is a 
contraction mapping on $W_x(J_T,r)$ for appropriately chosen $r, T,$ and $x$.\\

\medskip
\noindent\underline{\bf Claim 3:} For $r, T,$ and $\varepsilon$ chosen sufficiently small 
and positive, if $x \in \bar{B}_{E_{\mu}}(x_0, \varepsilon)$, then $\T_x$ maps 
$W_x(J_T,r)$ into itself.
\begin{subproof}[Proof of Claim 3]
	For $v \in W_x(J_T,r),$ notice that $\T_x(v) \in \E_{1,\mu}(J_T)$ and $\T_x(v)(0) = x,$
	by the property of maximal regularity and the definition of the mapping $\T_x$. 
	Thus, it suffices to show that 
	$\| \T_x(v) - u_{x_0}^\star \|_{\E_{1,\mu}(J_T)} \le r$ holds for all $v \in W_x(J_T, r)$,
	provided $r, T,$ and $\varepsilon$ are chosen sufficiently small.

	We begin with the observation	
	\begin{equation}\label{FirstSplit}
		\| \T_x(v) - u_{x_0}^\star \|_{\E_{1,\mu}(J_T)} \le
			\| \T_x(v) - u_{x}^\star \|_{\E_{1,\mu}(J_T)} + 
			\| u_x^\star - u_{x_0}^\star \|_{\E_{1,\mu}(J_T)}.
	\end{equation}
	Applying \eqref{ICStars} we get
	\[
		\| u_x^\star - u_{x_0}^\star \|_{\E_{1,\mu}(J_T)}
			\le C_1 |x - x_0|_{E_\mu},
	\]
	using the fact that $M_1\ge 1$.
	Note that choosing $\varepsilon$ sufficiently small, 
	this term can be bounded by $r / 4$ for all $x \in \bar{B}_{E_\mu}(x_0, \varepsilon)$.
	
	Since $(\T_x(v) - u_x^\star)\big|_{t=0} = 0$, we apply
	maximal regularity of $A(x_0)$ to bound the first term of \eqref{FirstSplit}
	\begin{equation}\label{MaxRegSplit}
	\begin{aligned}
		&\| \T_x(v) - u_{x}^\star \|_{\E_{1,\mu}(J_T)} \\ 
		&\quad \le C_{T_1} \big(
			  \| (A(x_0) - A(v))v \|_{\E_{0,\mu}(J_T)} + \| F_1(v) \|_{\E_{0,\mu}(J_T)} 
			+ \| F_2(v) \|_{\E_{0,\mu}(J_T)} \big), 
	\end{aligned}
	\end{equation}
	where $C_{T_1} > 0$ is the constant of maximal regularity for the interval $[0,T_1]$; 
	recalling that $T_1 > 0$ was introduced before bounds \eqref{StrongCont}--\eqref{ICStars}. 
	The first two terms of \eqref{MaxRegSplit} are bounded as in \eqref{ABound} and \eqref{F1Bound},
	respectively --- which are both bounded by $r/4C_{T_1}$ for $r, T,$ and $\varepsilon$
	sufficiently small.
	
	Addressing the last term in \eqref{MaxRegSplit}, we first split
	\[
		\|F_2(v) \|_{\E_{0,\mu}(J_T)} \le 
			\| F_2(v) - F_2(u_{x_0}^\star) \|_{\E_{0,\mu}(J_T)} + 
			\| F_2(u_{x_0}^\star) \|_{\E_{0,\mu}(J_T)},
	\]
	then note that $\| F_2(u_{x_0}^\star) \|_{\E_{0,\mu}(J_T)}$ can be made arbitrarily 
	small by taking $T$ sufficiently small. Meanwhile, we apply
	\eqref{Structure2} and Remark~\ref{InterpConstants} to bound 
\begin{equation}\label{F2vMinusux0Star}
\begin{aligned}
		\|&F_2(v) - F_2(u_{x_0}^\star)|_{\E_{0,\mu}(J_T)} \\
			&\le M_0\sum_{j=1}^m 
				\bigg[ T^{(1-\mu)(1-\alpha_j)} + T^{(1-\mu)(1-\rho_j\alpha - \alpha_j)}
				\Big( \|v \|_{C(J_T,E_\mu)}^{\rho_j(1-\alpha)}
				\|v \|_{\E_{1,\mu}(J_T)}^{\rho_j\alpha} \\
			& \qquad +\|u_{x_0}^\star \|_{C(J_T, E_\mu)}^{\rho_j(1-\alpha)}
				\|u_{x_0}^\star \|_{\E_{1,\mu}(J_T)}^{\rho_j\alpha} \Big)\bigg] \| v - u_{x_0}^\star \|_{C(J_T, E_\mu)}^{1-\alpha_j} 
				\| v - u_{x_0}^\star \|_{\E_{1,\mu}(J_T)}^{\alpha_j}\\
			&\le M_0 \sum_{j=1}^m \bigg[ T^{(1-\mu)(1-\alpha_j)} + 
				T^{(1-\mu)(1-\rho_j\alpha - \alpha_j)}
				\big(2M_3 r^{\rho_j \alpha}\big)  \bigg] M_2 r,
	\end{aligned}
	\end{equation}
	where $M_0:=C_{\varepsilon_0} C_0$
	and $M_2, M_3$ are constants chosen as follows. Applying part of
	\eqref{Claim1Control} and Young's inequality, we select $M_2 > 0$ so that 
	\begin{equation}\label{YoungsIneq}
	\begin{aligned}
		\| v - &u_{x_0}^\star \|_{C(J_T, E_\mu)}^{1-\alpha_j} 
			\| v - u_{x_0}^\star \|_{\E_{1,\mu}(J_T)}^{\alpha_j} \\
		&\le (1-\alpha_j)\| v -u_{x_0}^\star \|_{C(J_T,E_\mu)} + 
			\alpha_j \| v - u_{x_0}^\star \|_{\E_{1,\mu}(J_T)} \\
		&\le (1 - \alpha_j)(M_1 r + C_1 \varepsilon) + \alpha_j r \\
		&\le M_2 r,
	\end{aligned}
	\end{equation}
	for $\varepsilon \le r$.
	Likewise, applying Claim 1, we select $M_3 > 0$ so that
	\begin{equation}\label{LastIneq}
	\begin{aligned}
		\|v \|&_{C(J_T,E_\mu)}^{\rho_j(1-\alpha)}
			\|v \|_{\E_{1,\mu}(J_T)}^{\rho_j\alpha} \\
		&\le \big(\varepsilon_0 + |x_0|_{E_\mu} \big)^{\rho_j (1-\alpha)} 
			\big(r + \|u_{x_0}^\star \|_{\E_{1,\mu}(J_T)}\big)^{\rho_j \alpha} \\
		&\le M_3 r^{\rho_j \alpha},
	\end{aligned}
	\end{equation}
	for all $T > 0$  sufficiently small so that $\| u_{x_0}^\star \|_{\E_{1,\mu}(J_T)} \le r$. 
	
	Finally, note that all terms in \eqref{F2vMinusux0Star} have a linear factor of $r$ and an additional
	factor that can be made arbitrarily small by restricting the sizes of $r, T,$ and $\varepsilon$.
	In context of Remark~\ref{InterpConstants}(b), we note that terms involving subcritical
	index $j$ are made small with $T$ alone, while critical indices $j$ require restriction on
	the size of $r$.
	We conclude that the last term in \eqref{MaxRegSplit} can be bounded by
	$r / 4$ and \eqref{FirstSplit} can thus be bounded by $r$ for all $r, T,$ and 
	$\varepsilon$ chosen sufficiently small. This proves Claim~3.
\end{subproof}

\medskip
\underline{\bf Claim 4:} There exist constants $\kappa = \kappa(r,T,\varepsilon) > 0$ and
$\sigma = \sigma(r,T,\varepsilon) > 0$ so that for all $x_1, x_2 \in \bar{B}_{E_\mu}(x_0, \varepsilon)$ and 
$v_i \in W_{x_i}(J_T,r)$, $i = 1,2$, it holds that 
\[
	\| \T_{x_1}(v_1) - \T_{x_2}(v_2) \|_{\E_{1,\mu}(J_T)} \le 
		\kappa \| v_1 - v_2 \|_{\E_{1,\mu}(J_T)} + \sigma | x_1 - x_2 |_{E_\mu}.
\] 
Further, $\kappa$ is made arbitrarily small by choosing $r, T,$ and $\varepsilon$
sufficiently small.\\
\begin{subproof}[Proof of Claim 4]
	Let $x_1, x_2 \in \bar{B}_{E_\mu}(x_0, \varepsilon)$ be given and pick 
	$v_1 \in W_{x_1}(J_T,r)$, $v_2 \in W_{x_2}(J_T,r)$. By \eqref{TraceZero} and \eqref{ICStars}
	we have
	\begin{align}\label{Claim4Central}
		\| v_1 - v_2 \|_{C(J_T,E_\mu)} &\le 
			\| (v_1 - v_2) - (u_{x_1}^\star - u_{x_2}^\star) \|_{C(J_T,E_\mu)} + 
			\|u_{x_1}^\star - u_{x_2}^\star \|_{C(J_T,E_\mu)} \notag\\
		&\le M_1 \| v_1 - v_2 \|_{\E_{1,\mu}(J_T)}
			+ C_1 |x_1 - x_2|_{E_\mu}.
	\end{align}
	
	Proceeding, we first note that 
	$\big((\T_{x_1}(v_1) - \T_{x_2}(v_2)) - (u_{x_1}^\star - u_{x_2}^\star)\big)\big|_{t = 0} = 0,$
	so we compute, applying maximal regularity of $A(x_0)$ and \eqref{ICStars},
	\begin{align}\label{Claim4Split}
		\|\T_{x_1} &(v_1) - \T_{x_2} (v_2)\|_{\E_{1,\mu}(J_T)} \notag\\
			&\le \|(\T_{x_1}(v_1) - \T_{x_2}(v_2) ) - 
				(u_{x_1}^\star - u_{x_2}^\star)\|_{\E_{1,\mu}(J_T)}
				+ \| u_{x_1}^\star - u_{x_2}^\star \|_{\E_{1,\mu}(J_T)} \notag\\
			&\le C_{T_1} \Big[ \| (A(v_2) - A(v_1)) v_2 \|_{\E_{0,\mu}(J_T)} +	\| (A(x_0) - A(v_1))(v_1 - v_2) \|_{\E_{0,\mu}(J_T)}\notag\\
			&\quad + \| F_1(v_1) - F_1(v_2) \|_{\E_{0,\mu}(J_T)} +
				\|F_2(v_1) - F_2(v_2)\|_{\E_{0,\mu}(J_T)} \Big] \\
			&\quad +	C_1 |x_1 - x_2|_{E_\mu}. \notag
	\end{align}
	Continuing with individual terms in \eqref{Claim4Split},
	we apply \eqref{Structure1} and \eqref{Claim4Central} to get
	\begin{align*}
				\| (A(v_2) &- A(v_1)) v_2 \|_{\E_{0,\mu}(J_T)} \\
			&\le L \|v_1 - v_2\|_{C(J_T,E_\mu)} \|v_2\|_{\E_{1,\mu}(J_T)} \\
			&\le L \Big(M_1 \| v_1 - v_2 \|_{\E_{1,\mu}(J_T)}
			+ C_1 |x_1 - x_2|_{E_\mu}\Big) \big(r + \|u_{x_0}^\star\|_{\E_{1,\mu}(J_T)}\big),\\
			\| F_1(v_1)& - F_1(v_2) \|_{\E_{0,\mu}(J_T)} \\
			& \le T^{1-\mu}L\Big(M_1 \| v_1 - v_2 \|_{\E_{1,\mu}(J_T)}
			+ C_1 |x_1 - x_2|_{E_\mu}\Big), 
      \end{align*}
	and, also applying \eqref{Claim1Control}, we have
	\begin{align*}
		\|(A(x_0) &- A(v_1))(v_1 - v_2)\|_{E_{0,\mu}(J_T)}\\
			&\le L \|v_1 - x_0\|_{C(J_T,E_\mu)} \|v_1 - v_2\|_{\E_{1,\mu}(J_T)}\\
			&\le L\Big(M_1 r + C_1 \varepsilon + \|u_{x_0}^\star - x_0 \|_{C(J_T, E_\mu)}\Big)
				\|v_1 - v_2\|_{\E_{1,\mu}(J_T)}.
	\end{align*}
	
	Meanwhile, by Young's inequality and \eqref{Claim4Central}, recalling 
	that $M_1 \ge 1$,	we have
	\begin{align*}\label{YoungsPlus}
		\|v_1 - &v_2\|_{C(J_T,E_\mu)}^{1-\alpha_j} 
				\|v_1 - v_2\|_{\E_{1,\mu}(J_T)}^{\alpha_j} \notag\\
			&\le M_1 \|v_1 - v_2 \|_{\E_{1,\mu}(J_T)} + C_1 |x_1 - x_2|_{E_\mu},
	\end{align*}
	which we apply in combination with \eqref{Structure2}, Remark~\ref{InterpConstants}, and 
	\eqref{LastIneq} to bound
	\begin{align*}
		\|&F_2(v_1) - F_2(v_2) \|_{\E_{0,\mu}(J_T)} \\
			&\le  M_0 \sum_{j=1}^m 
				\bigg[ T^{(1-\mu)(1-\alpha_j)}
				+ T^{(1-\mu)(1-\rho_j\alpha - \alpha_j)}
				\Big( \|v_1 \|_{C(J_T,E_\mu)}^{\rho_j(1-\alpha)}
				\|v_1 \|_{\E_{1,\mu}(J_T)}^{\rho_j\alpha} \notag\\
			& \quad +\|v_2 \|_{C(J_T, E_\mu)}^{\rho_j(1-\alpha)}
				\|v_2 \|_{\E_{1,\mu}(J_T)}^{\rho_j\alpha} \Big)\bigg]
				\| v_1 - v_2 \|_{C(J_T, E_\mu)}^{1-\alpha_j} 
				\| v_1 - v_2 \|_{\E_{1,\mu}(J_T)}^{\alpha_j} \notag\\
			&\le M_0 \sum_{j=1}^m 
				\bigg[ T^{(1-\mu)(1-\alpha_j)}
				+ T^{(1-\mu)(1-\rho_j\alpha - \alpha_j)}
				\big( 2 M_3 r^{\rho_j\alpha} \big)\bigg] \notag\\
			& \quad \cdot \big( M_1 \|v_1 - v_2 \|_{\E_{1,\mu}(J_T)} + 
				 C_1 |x_1 - x_2|_{E_\mu} \big), \notag\\
	\end{align*}
	with  $M_0:=C_{\varepsilon_0} C_0$.
	Combining all terms involving $\| v_1 - v_2 \|_{\E_{1,\mu}(J_T)}$ and likewise
	terms involving $|x_1 - x_2|_{E_\mu}$, note that \eqref{Claim4Split} 
	takes on the desired structure for the claim. 
	Moreover, every factor multiplying the terms $\| v_1 - v_2 \|_{\E_{1,\mu}(J_T)}$
	can be made arbitrarily small by taking either $r, T,$ or $\varepsilon$ sufficiently small.
	Note that the same cannot be said for every factor of $|x_1 - x_2|_{E_\mu}$, 
	as seen in the last term of \eqref{Claim4Split}. 
	Regardless, we have thus proved Claim 4.
\end{subproof}

Finally, fix $r, T,$ and $\varepsilon$ small enough so that $\kappa \le \frac{1}{2}$.
We thus have the estimate
\begin{equation}\label{3/4}
\| \T_{x_1}(v_1) - \T_{x_2}(v_2) \|_{\E_{1,\mu}(J_T)} \le \frac{1}{2} \| v_1 - v_2 \|_{\E_{0,\mu}(J_T)} + \sigma |x_1-x_2|_{E_\mu},
\end{equation}
for every $ v_i\in W_{x_i}(J_T, r)$ and $ x_i \in \bar B_{E_\mu}(x_0,\varepsilon)$.
Let $x_1=x_2=x\in B_{E_\mu}(x_0,\varepsilon)$ be given. Then 
\[
	\| \T_x(v_1) - \T_x(v_2) \|_{\E_{1,\mu}(J_T)} \le \frac{1}{2} \| v_1 - v_2 \|_{\E_{0,\mu}(J_T)}
	\qquad \text{for $v_1, v_2 \in W_x(J_T,r)$}, 
\]
and so $\T_x$ is a strict contraction on $W_x(J_T,r)$. Applying Banach's fixed point theorem,
we obtain a unique fixed point
\[
	u(\cdot, x) \in W_x(J_T,r) \subset \E_{1,\mu}(J_T), \qquad 
		\text{for every $x \in \bar{B}_{E_\mu}(x_0, \varepsilon),$}
\]
which solves \eqref{QuasiEqn} by construction of the mapping $\T_x$.
Furthermore, for $x_1, x_2 \in \bar{B}_{E_\mu}(x_0, \varepsilon)$, \eqref{3/4} implies
\[
	\|u(\cdot, x_1) - u(\cdot, x_2) \|_{\E_{1,\mu}(J_T)} \le 2\sigma |x_1 - x_2|_{E_\mu},
\] 
which completes the proof of the first statement of the theorem.

\medskip
(b) By a standard argument, we can extend the local solution
obtained in part (a) to a maximal solution on some right--open interval $[0, t^+(x_0)).$ 
To confirm this maximal solution satisfies the stated regularity,
we consider a portion of this extension argument. In particular, with $x_0 \in V_\mu$
given, we apply part (a) to produce the solution $u_1(\cdot, x_0) \in \E_{1, \mu}([0,\tau_1])$ 
on some interval $[0,\tau_1].$ Then, we note that $x_1 := u_1(\tau_1, x_0) \in E_1 \cap V_\mu$, 
and so we may apply part (a) again to produce the solution 
$u_2(\cdot, x_1) \in \E_{1,\mu}([0,\tau_2])$ on a second interval
$[0,\tau_2].$ It follows that
\[
	u(t) := \begin{cases}
			u_1(t), &\text{for $t \in [0,\tau_1]$}\\
			u_2(t-\tau_1), &\text{for $t \in [\tau_1, \tau_1+\tau_2]$}
		\end{cases}
\]
satisfies \eqref{QuasiEqn} with $u(0) = x_0$ and regularity $u \in \E_{1,\mu}([0,\tau_1 + \tau_2]).$
To prove this last claim, it suffices to show that $u_2 \in C([0,\tau_2], E_1),$ in particular
$\lim_{t \to 0^+} u_2(t) = x_1.$ For that purpose, we fix $\varepsilon > 0$ so that the result
of part (a) holds for $x \in \bar B_{E_\mu} (x_1, \varepsilon)$ and choose $\delta \in (0, \tau_2)$ 
sufficiently small that $u_1(\tau_1 - \delta) \in \bar B_{E_\mu} (x_1, \varepsilon).$ Now let
$v \in \E_{1,\mu}([0,\tau_2])$ denote the solution to \eqref{QuasiEqn} with initial value
$u_1(\tau_1 - \delta).$ By uniqueness of solutions, it follows that 
\[
	v\big|_{[0,\delta]} = u_1\big|_{[\tau_1 - \delta, \tau_1]} 
	\qquad \text{and}	\qquad
	v\big|_{[\delta,\tau_2]} = u_2\big|_{[0,\tau_2-\delta]},
\]
and the desired regularity of $u_2$ now follows by the regularity of $v \in \E_{1,\mu}([0,\tau_2])$
away from $t =0$.

\medskip
(c) The proof of global existence follows exactly as in \cite[Theorem 4.1 (c)--(d)]{CS01},
with the regularity of maximal solutions confirmed in (b) above.
\end{proof}
\goodbreak
We conclude this section on well--posedness with the following
extension of \eqref{Lip-mu}, accounting for the dependence of solutions on 
initial data residing in smaller spaces $E_{\bar \mu} \subset E_\mu$.
This result will be useful in the following section as we consider long--term 
dynamics of solutions that start in $E_\mu$ and instantaneously regularize
to spaces $E_{\bar \mu}$.

\begin{prop}\label{existence-bar}
Suppose the assumptions of Theorem~\ref{thm:MainResult} hold, 
$\bar \mu \in [\beta,1)$, and $x_0\in V_\mu\cap E_{\bar\mu}$.
Then Theorem~\ref{thm:MainResult}(a) holds true with $\mu$ replaced by $\bar\mu$. In particular,
\begin{equation}\label{Lip-bar-mu}
\| u(\cdot, x_1) - u(\cdot, x_2) \|_{\E_{1, \bar\mu}([0,\tau])}  \le \sigma | x_1 - x_2 |_{E_{\bar\mu}},
			\quad  x_1, x_2 \in \bar{B}_{E_{\bar\mu}}(x_0, \varepsilon).
\end{equation}			
\end{prop}
\begin{proof} 
We recall that $ \mathcal M_\mu(E_1,E_0)\subset \mathcal M_{\bar\mu}(E_1,E_0),$
see~\cite[Lemma 2.6]{CS01}.
Hence it follows from  {\bf (H1)}-{\bf (H2)} that
\begin{equation}
(A,F)\in C^{1-}\big (V_\mu\cap E_{\bar\mu},\;\mathcal{M}_{\bar\mu}(E_1, E_0) \times E_0 \big),
\quad \text{where}\;\; F:=F_1+F_2.
\end{equation}
Existence of a unique solution $u_{\bar\mu}=u_{\bar\mu}(\cdot,x)\in \E_{1,\bar\mu}([0,\tau])$
with  property~\eqref{Lip-bar-mu}
follows as in the proof of Theorem~\ref{thm:MainResult}(a), with $F_1=F$ and $F_2=0$.
In both cases, the solution is obtained as a fixed point of a strict contraction
$\T_x: {\sf M}_\nu\to {\sf M}_\nu$,  
where ${\sf M}_\nu$ is a closed subset of $\E_{1,\nu}([0,\tau])$, 
respectively,  with $\nu\in\{\mu,\bar\mu\}$.
But 
$$\T_x: {\sf M}_\mu \cap {\sf M}_{\bar\mu}\to {\sf M}_\mu \cap {\sf M}_{\bar\mu}$$ 
is a strict contraction as well, and thus has a unique fixed point
$u_\star\in {\sf M}_\mu \cap {\sf  M}_{\bar\mu}$. 
Therefore, $u_\mu=u_{\bar\mu}=u_\star$ on $[0,\tau]$,
where $u_\mu=u_\mu(\cdot,x)\in\E_{1,\mu}([0,\tau])$ is 
the solution obtained in Theorem~\ref{thm:MainResult}(a).
This shows, in particular, that each solution $u(\cdot,x)$, with 
$x\in B_{E_{\bar\mu}}(x_0,\varepsilon)$, obtained in Theorem~\ref{thm:MainResult}(a), 
also belongs to $\E_{1,\bar\mu}([0,\tau])$.
\end{proof}

\section{Normal Stability}\label{sec:Stability}
With well--posedness of \eqref{QuasiEqn} established, 
we investigate the long--term behavior of solutions that start near equilibria. 
In particular, in this section we demonstrate that the so--called
generalized principle of linearized stability (c.f. \cite{PSZ09b, PSZ09a}) 
continues to hold on $E_\mu$, provided the pertinent assumptions are satisfied.
As a first step in this direction, we prove that stability of 
equilibria can be tracked in either the topology of $E_\mu$ or, equivalently, 
in the stronger topology of $E_{\bar \mu}.$

\begin{prop}\label{pro:stability}
Suppose the assumptions of Theorem~\ref{thm:MainResult} hold, $\bar \mu \in [\beta,1)$, and suppose
$u_*\in V_\mu\cap E_1$ is an equilibrium for~\eqref{QuasiEqn}. Then \\
 $u_*$ is stable in the topology of $E_\mu$ $\Longleftrightarrow$
 $u_*$ is stable in the topology of $E_{\bar\mu}$.
\end{prop}
\begin{proof} 
By Theorem~\ref{thm:MainResult}(a) and Proposition~\ref{existence-bar}
there are constants  $\tau=\tau(u_*)$, $\eta=\eta(u_*)$ and $c_1=c_1(u_*)$, corresponding to the
initial value $u_*$, such that
\begin{equation}\label{AA}
\|u(\cdot, y_0)-u_*\|_{\E_{1,\nu}([0,2\tau])}\le c_1 |y_0-u_*|_{E_\nu},
\quad \nu\in\{\mu,\bar\mu\},
\end{equation}
for any $y_0\in B_{E_\nu}(u_*,\eta)$.
Moreover,  one readily verifies that there is a constant $c_2=c_2(\tau,\mu,\bar\mu)$ such that
\begin{equation}\label{AB}
\| v - u_*\|_{BC([\tau,2\tau], E_{\bar\mu})} \le c_2 \| v -u_*\|_{\E_{1,\mu}([0,2\tau])} 
\end{equation}
for any function $v\in \E_{1,\mu}([0,2\tau])$.
 In the sequel, we denote the embedding constant of $E_{\bar\mu}\hookrightarrow E_{\mu}$
by $c_\mu$. Consequently,
\begin{equation}\label{AC}
B_{E_{\bar\mu}}(u_*, \alpha)\subset B_{E_{\mu}}(u_*, c_\mu\alpha).
\end{equation}
\medskip\noindent
Suppose $u_*$ is stable in $E_{\mu}$. 
Let $\varepsilon>0$ be given and set $\varepsilon_\mu:=\min\{\varepsilon/(c_1c_2),\eta\}.$
By assumption, there is a number $\delta_\mu$ 
such that every solution of~\eqref{QuasiEqn} with initial value $x_0\in B_{E_\mu}(u_*,\delta_\mu)$
exists globally and satisfies
\begin{equation}\label{AD}
 |u(t,x_0)-u_*|_{E_\mu}<\varepsilon_\mu, \quad\text{for all} \;\ t\ge 0.
 \end{equation}
Next, we choose $\delta \in (0,\delta_\mu/c_\mu]$ sufficiently small such that
\begin{equation}\label{AE}
|u(t,x_0)-u_*|_{E_{\bar\mu}}< \varepsilon,\quad
	\text{for all} \;\; t\in [0,\tau],\;\; x_0\in B_{E_{\bar\mu}}(u_*,\delta).
\end{equation}
Here we note that \eqref{AE} follows from continuous dependence on the initial data, see \eqref{AA}.
As a consequence of \eqref{AC}, every solution $u(\cdot, x_0)$ with 
$x_0\in B_{E_{\bar\mu}}(u_*,\delta)$ exists globally
and satisfies \eqref{AD} as well as \eqref{AE}.
Next we will show by induction that
 $u(t,x_0)\in B_{E_{\bar\mu}}(u_*,\varepsilon)$ for all $t\ge 0$.
 Suppose we have already shown that
 $|u(t,x_0)-u_*|_{E_{\bar\mu}}< \varepsilon$ for $t\in [0,(k+1)\tau]$ and $k\in\mathbb N.$
 We note that the case $k=0$ is exactly \eqref{AE}.
From the definition of $\varepsilon_\mu$ and \eqref{AA}-\eqref{AB} as well as \eqref{AD} follows
 \begin{equation}
	|u(k\tau +s,x_0)-u_*|_{E_{\bar\mu}}\le 
 		c_1c_2 |u(k\tau,x_0) -u_*|_{E_\mu}<\varepsilon,\quad \tau\le s\le 2\tau. 
 \end{equation}
 Since this step works for any $k\in\mathbb N$, we obtain stability of $u_*$ in $E_{\bar\mu}$.
 \medskip\\
 \noindent
Suppose that $u_*$ is stable in $E_{\bar\mu}$. Let $\varepsilon>0$ be given
 and set $\varepsilon_{\bar\mu}= \varepsilon/c_\mu$.  
 By the stability assumption, there exists a number $\delta_{\bar\mu}$ such that
every solution of \eqref{QuasiEqn} with initial value $x_0\in B_{E_{\bar\mu}}(u_*,\delta_{\bar\mu})$
exists globally and satisfies
\begin{equation}\label{AF}
 |u(t,x_0)-u_*|_{E_{\bar\mu}}<\varepsilon_{\bar\mu}, \quad\text{for all} \;\ t\ge 0.
 \end{equation}
Next, by continuous dependence on initial data, 
we can choose $\delta\in (0,\eta)$ sufficiently small such that
\begin{equation}\label{AG}
	|u(t,x_0)-u_*|_{E_\mu}< \delta_{\bar\mu}/(c_1c_2),\quad
		\text{for all} \;\; t\in [0,\tau],\;\; x_0\in B_{E_\mu}(u_*,\delta).
\end{equation}
It follows from \eqref{AA}-\eqref{AB} and \eqref{AG} that
$|u(\tau,x_0) - u_*|_{E_{\bar\mu}} \le c_1c_2 |x_0-u_*|_{E_\mu}<\delta_{\bar\mu}$,
for all  $x_0\in B_{E_{\bar\mu}}(u_*,\delta)$.
In particular, after a short time, we are simply tracking the solutions
$u(\cdot, x_0)$ in the stronger topology of $E_{\bar\mu}$.
Hence, by \eqref{AC} and \eqref{AF},
 $u(t,x_0)\in B_{E_\mu}(u_*,\varepsilon)$ for
any initial value $x_0\in B_{E_{\mu}}(u_*,\delta)$.
This completes the proof of Proposition~\ref{pro:stability}.
\end{proof}

In addition to {\bf (H1)}-{\bf (H2) }we now assume that 
\begin{equation}
\label{A-F1-F2}
(A,F_1,F_2)\in C^1(V_\mu\cap E_{\bar\mu} ,\;\mathcal L(E_1,E_0)\times E_0\times E_0),
\end{equation}
where  $\bar\mu\in [\beta,1)$ is a fixed number.
Here we note that $V_\mu\cap E_{\bar\mu}\subset E_{\bar\mu}$ is open,
and that differentiability is understood with respect to the topology of $E_{\bar\mu}$.
\smallskip\\ \noindent
Let $ \cE\subset V_\mu\cap E_1$ denote the set of equilibrium solutions 
of \eqref{QuasiEqn}, which means that
$$
u\in\cE \;\text{ if and only if }\; u\in V_\mu\cap E_1 \; \text{ and } \; A(u)u=F_1(u) +F_2(u).
$$
Given an  element $u_*\in\cE$,  we assume that $u_*$ is
contained in an $m$-dimensional manifold of equilibria. This means that there
is an open subset $U\subset\R^m$, $0\in U$, and a $C^1$-function
$\Psi:U\rightarrow E_1$,  such that
\begin{equation}
\label{manifold}
\begin{aligned}
& \bullet\
\text{$\Psi(U)\subset \cE$ and $\Psi(0)=u_*$,} \\
& \bullet\
 \text{the rank of $\Psi^\prime(0)$ equals $m$, and} \\
& \bullet\
\text{$A(\Psi(\zeta))\Psi(\zeta)=F(\Psi(\zeta)),\quad \zeta\in U.$}
\end{aligned}
\end{equation}
We assume furthermore that near $u_*$ there are no other equilibria
than those given by $\Psi(U)$,
i.e.\ $\cE\cap B_{E_1}(u_*,{r_1})=\Psi(U)$, for some $r_1>0$.
\smallskip\\ \noindent
For $u_*\in\cE$, we define
\begin{equation}\label{A0}
A_0v = A(u_*)v+(A^\prime(u_*)v)u_* - F_1^\prime(u_*)v - F_2^\prime(u_*)v, \quad v\in E_1,
\end{equation}
where $A^\prime, F_1^\prime$ and $F_2^\prime$ denote the 
Fr\'echet derivatives of the respective functions.
We denote by $N(A_0)$ and $R(A_0)$ the kernel and range, respectively, of the 
operator $A_0.$
\medskip\\
\noindent
After these preparations we can state the following result on 
convergence of  solutions starting near $u_*$.
\goodbreak
\begin{thm} 
\label{thm:normally-stable}
Suppose $u_*\in V_\mu\cap E_1$ is an
equilibrium of~\eqref{QuasiEqn}, and suppose that the functions
$(A,F_1,F_2)$ satisfy  {\bf(H1)}-{\bf (H2)} as well as \eqref{A-F1-F2}.
Finally, suppose that $u_*$ is normally stable, i.e.,
\begin{enumerate}
\item[(i)] near $u_*$ the set of equilibria $\cE$ is a $C^1$-manifold 
in $E_1$ of dimension $m\in \mathbb N$,
\item[(ii)] \, the tangent space for $\cE$ at $u_*$ is given by $N(A_0)$,
\item[(iii)] \, $0$ is a semi-simple eigenvalue of
$A_0$, i.e.,\ $ N(A_0)\oplus R(A_0)=E_0$,
\item[(iv)] \, $\sigma(-A_0)\setminus\{0\}\subset \{z\in \mathbb C:\, {\rm Re}\, z<0\}$.
\end{enumerate}
Then $u_*$ is stable in $E_\mu$. Moreover, there exists a constant $\delta=\delta(\bar\mu)>0$ such
that each solution $u(\cdot,x_0)$ of \eqref{QuasiEqn} with initial
value $x_0\in B_{E_\mu}(u_*,\delta)$ exists globally and converges 
to some $u_\infty\in\cE$ in  $E_{\bar\mu}$ at an exponential rate as $t\to\infty$.
\end{thm}
\begin{proof}
Example 2 and Theorem 3.1 in \cite{PSZ09a} imply  stability of $u_*$ in $E_{\bar\mu}$.
Moreover, the same theorem ensures that there exists $\delta_1>0$ such that
each solution $u(\cdot,y_0)$ of \eqref{QuasiEqn} with initial value $y_0\in B_{E_{\bar\mu}}(u_*,\delta_1)$ exits globally and 
converges to some $u_\infty\in\cE$ in the topology of $E_{\bar\mu}$ at an exponential rate.

By Proposition~\eqref{pro:stability},  $u_*$ is stable in $E_\mu$ as well.
Employing \eqref{AA}-\eqref{AB}, we deduce that there exists $\delta=\delta(\delta_1,\bar\mu)>0$ such that
\begin{equation*}
|u(\tau,x_0)-u_*|_{E_{\bar\mu}}<\delta_1 \; \text{ for each }\; x_0\in B_{E_\mu}(u_*,\delta).
\end{equation*}
As $u(\tau,x_0)\in E_{\bar\mu}$ and $u(t, u(\tau,x_0))=u(t+\tau,x_0)$ for $t\ge\tau$, the convergence assertion of the Theorem
follows from the first part of the proof.
\end{proof}
\begin{remark}
Theorem~\ref{thm:normally-stable} yields convergence 
of $u(\cdot,x_0)$ in the stronger norm of $E_{\bar\mu}$ for initial values in $E_\mu$.
We note that this holds true for any  $\bar\mu\in [\beta,1)$,
with $\delta$ depending on $\bar\mu$.
\end{remark}

\section{Applications to Surface Diffusion Flow}\label{sec:SDFlow}
In this section, we apply the theory from the previous sections to extend results 
regarding the surface diffusion flow in various settings. 
First, we extend \cite[Proposition 3.2]{LS16} regarding well--posedness of the 
surface diffusion flow in the setting of so--called \textit{axially--definable} surfaces. 
We then prove nonlinear stability of cylinders with radius $r > 1$ (as equilibria of surface diffusion flow) 
under a general class of periodic perturbations which only require control of first--order derivatives;
this result extends \cite[Theorem 4.3]{LS16} and \cite[Theorem 4.9]{LS13}
where control of second--order derivatives was also required.
At the conclusion of the section, we establish general well--posedness for
surface diffusion flow acting on surfaces parameterized over 
a compact reference
manifold $\Sigma \subset \R^n$, 
and conclude normal stability of Euclidean spheres under $bc^{1+\alpha}$ perturbations.

\subsection{Axially--Definable Setting: Well--Posedness}\label{subsec:AxiallyDefinable}
We begin with a brief introduction to the axially--definable setting and formulation of the problem;
for a more detailed account we direct the reader to \cite[Sections 2 and 3]{LS16}.

First, given $r > 0$, let 
\[
	\mathcal{C}_r := \{ \big(x, r\cos(\theta), r\sin(\theta)\big): x \in \R, \theta \in \mathbb{T}\}
\]
denote the unbounded cylinder in $\R^3$ of radius $r$, where 
$\mathbb{T} := [0,2\pi]$ denotes the one--dimensional torus, with
$0$ and $2\pi$ identified.
Next, we fix a parameter $\alpha \in (0,1)$ and define the Banach spaces
\begin{equation}\label{ESpaces}
	E_0 := bc^\alpha(\mathcal{C}_r) \qquad \text{and} \qquad
	E_1 := bc^{4 + \alpha}(\mathcal{C}_r),
\end{equation}
where $bc^{k + \alpha}$, $k \in \mathbb{N}$, denotes the family of $k$--times differentiable
\textit{little--H\"older} regular functions.
In particular, on an open set $U \in \R^n$, $bc^{\alpha}(U)$ is defined as the closure of the
bounded smooth functions $BC^{\infty}(U)$ in the topology of $BC^{\alpha}(U),$ the Banach
space of all bounded H\"older--continuous functions of exponent $\alpha$. 
Then $bc^{k+\alpha}(U)$ consists of functions having continuous and bounded derivatives 
of order $k$, whose $k^{\rm th}$--order derivatives are in $bc^{\alpha}(U)$.
The space $bc^{k+\alpha}(\mathcal{C}_r)$ is defined via an atlas of local charts.

Taking $\mu = 1/4$ and $\beta = 3/4$, we define the continuous interpolation spaces
$E_\mu := (E_0, E_1)_{\mu, \infty}^0$ and $E_\beta := (E_0, E_1)_{\beta, \infty}^0.$
It is well--known that the scale of little--H\"older spaces is closed under continuous
interpolation (c.f. \cite{Lun95} and \cite{SS14}) and so these spaces are likewise identified as
\[
	E_\mu = bc^{1+\alpha}(\mathcal{C}_r) \qquad \text{and} \qquad
	E_\beta = bc^{3+\alpha}(\mathcal{C}_r).
\]
With the spaces $E_0, E_\mu, E_\beta, E_1$ thus set, note that condition 
\eqref{Criticality} becomes
\[
	\frac{\rho_j}{2} + \beta_j \le 1,
\]
so that we have a critical index $j$ exactly when $\rho_j/2 + \beta_j = 1$.
Further, with $\varepsilon > 0$ fixed, we define the family of admissible initial values 
(which coincides with surfaces that remain bounded away from the central axis of rotation)
\[
	V_\mu := E_\mu \cap \{ h: \mathcal{C}_r \to \R \ |\  h(p) > \varepsilon - r 
		\text{ for all $p \in \mathcal{C}_r$} \}.
\]

We say that a surface $\Gamma \subset \R^3$ is \textit{axially--definable} if it can be
parameterized as
\[
	\Gamma = \Gamma(h) = \{ p + h(p)\nu(p) : p \in \mathcal{C}_r \}
\]
for some \textit{height function} $h : \mathcal{C}_r \to \R$ satisfying
$h > -r$ on $\mathcal{C}_r$,
where $\nu$ denotes the outer unit normal field over $\mathcal{C}_r$.
In the setting of axially--definable surfaces, the surface diffusion flow is expressed  
as the following evolution equation for time--dependent height functions $h = h(t, p) = h(t,x,\theta)$: 
\begin{equation}\label{SDFlow}
\begin{cases}
	h_t(t,p) = [G(h(t))](p), &\text{for $t > 0$, $p \in \mathcal{C}_r$}\\
	h(0) = h_0, &\text{on $\mathcal{C}_r$.}
\end{cases}
\end{equation}
As shown in \cite[Section 2.2]{LS16}, the evolution operator $G$ takes the form
\begin{equation}\label{SDOperator}
\begin{aligned}
	G(h) := 
		\frac{1}{(r + h)} &\left\{ \partial_x \left[ \frac{(r + h)^2 + h_\theta^2}
				{\sqrt{\mathcal{G}}} \ \partial_x \mathcal{H}(h) 
		- \frac{h_x h_\theta}{\sqrt{\mathcal{G}}} \ 
				\partial_\theta \mathcal{H}(h) \right] \right.\\
	& \quad + \ \partial_\theta \left.\left[ 
		\frac{1 + h_x^2}{\sqrt{\mathcal{G}}} \ \partial_\theta 	
			\mathcal{H}(h) 
		- \frac{\rho_x h_\theta}{\sqrt{\mathcal{G}}} \ \partial_x \mathcal{H}(h) 
			\right] \right\},
\end{aligned}
\end{equation}
where $\mathcal{H}(h)$ denotes the mean curvature of the surface $\Gamma(h)$
and $\mathcal{G} = \mathcal{G}(h)$ is the
determinant of the first fundamental form $[g_{ij}]=[g_{ij}(h)]$ on $\Gamma(h)$.

Using \cite[Equations (2.2)--(2.3)]{LS16},
one can expand \eqref{SDOperator} to see that $G(h)$ is a fourth--order
quasilinear operator of the form
\[
\begin{aligned}
	G(h) &= -A(h) h + F_1(h) + F_2(h) \\
		&:= - \left( \sum_{|\eta| = 4} b_\eta(h, \partial^1 h) \ \partial^\eta h \right) 
			+ F_1(h, \partial^1 h) 
			+ F_2(h, \partial^1 h, \partial^2 h, \partial^3 h),
\end{aligned}
\]
where $\eta = (\eta_1, \eta_2) \in \mathbb{N}^2$ is a multi--index, $|\eta| := \eta_1 + \eta_2$ 
its length, $\partial^\eta := \partial_x^{\eta_1}\partial_\theta^{\eta_2}$ the mixed 
partial derivative operator, and $\partial^k h$ denotes the vector of all 
derivatives $\partial^\eta h$ for $|\eta| = k.$
We note that $A(h)$ here contains only the highest--order terms of the operator
$\mathcal{A}(h)$ expressed in \cite[Section 3.2]{LS16} --- which is essential in the current
setting to ensure the coefficients $b_\eta(h, \partial^1 h)$ are well--defined for 
$h \in E_\mu = bc^{1 + \alpha}(\mathcal{C}_r)$.
It follows that the principal symbols $\sigma[A(h)]$ and $\sigma[\mathcal{A}(h)]$
coincide, so by \cite[Equation~(3.3)]{LS16} we have
\begin{equation}\label{A-Elliptic}
	\sigma[A(h)](p,\xi) \ge \frac{1}{\mathcal{G}^2} 
		\bigg( (r+h)^2 \xi_1^2 + \xi_2^2 \bigg)^2 \qquad \text{for} \quad
		(p,\xi) \in \mathcal{C}_r \times \R^2.
\end{equation}
This last result implies uniform ellipticity of $A(h)$ on $\mathcal{C}_r$
and thus with \cite[Proposition~3.1]{LS16} we have
\begin{equation}\label{A-F1-F2-regularity}
\begin{aligned}
	(A,F_1) &\in C^{\omega} \big( V_\mu, \mathcal{M}_\mu (E_1, E_0) \times E_0 \big)\\
	F_2 &\in C^{\omega} \big( V_\mu \cap E_\beta, E_0 \big).
\end{aligned}
\end{equation}

We have now confirmed that the mappings $A, F_1$ and $F_2$ satisfy properties 
{\bf (H1)} and \eqref{A-F1-F2}.
Regarding confirmation of the structural conditions {\bf (H2)}, we expand terms 
of \eqref{SDOperator} to confirm
\begin{equation}\label{F2Structure}
\begin{aligned}
	F_2(h) &= \sum_{\substack{|\eta| = 3\\|\tau| \le 2}} 
		c_{\eta, \tau}(h) \, \partial^\tau h \, \partial^\eta h\\
		&\qquad 
		+ \sum_{\max\{|\eta|, |\tau|, |\sigma|\}=2 } d_{\eta, \tau, \sigma}(h) 
		\, \partial^\eta h \, \partial^\tau h \, \partial^\sigma h,
\end{aligned}
\end{equation}
where the functions $c_{\eta, \tau}, \ d_{\eta, \tau, \sigma}$ depend only upon $h$ 
and $\partial^1 h$, and are analytic by 
\eqref{A-F1-F2-regularity}.
Of particular importance in \eqref{F2Structure}, we note that third--order derivatives
of $h$ appear linearly in terms with at most linear factors of $\partial^2 h$,
while lower--order terms include at most cubic factors of $\partial^2 h.$

Letting $h_0 \in V_\mu$ and $R > 0$, we choose 
$h_1, h_2 \in \bar{B}_{E_\mu}(h_0, R) \cap (V_\mu \cap E_\beta)$ and 
use \eqref{F2Structure} to bound $|F_2(h_1) - F_2(h_2)|_{E_0}$.
Throughout the following computations, we use $\tilde{C}$ to denote a
generic constant that depends only upon $R$ and $|h_0|_{E_\mu} = |h_0|_{bc^{1+\alpha}}$.

Considering the first term in \eqref{F2Structure}, when $|\tau| \le 1$ 
we incorporate $\partial^\tau h_i$ into $c_{\eta, \tau}(h_i)$ to derive bounds depending 
only on zeroth and first--order derivatives, $i=1,2$. Hence, when $|\tau| \le 1$ we have
\[
	|c_{\eta, \tau}(h_1) \, \partial^\eta h_1\, \partial^\tau h_1 - 
		c_{\eta, \tau}(h_2) \, \partial^\eta h_2\, \partial^\tau h_2|_{E_0} \le
		\tilde{C}|h_1 - h_2|_{bc^{3+\alpha}},
\]
which is a term as in \eqref{Structural} corresponding to $(\rho_j, \beta_j) = (0, 3/4)$
(which is a subcritical index since $\rho_j/2 + \beta_j < 1$). 
Meanwhile, when $|\tau| = 2$, we rewrite
\begin{equation}\label{cDifference}
\begin{aligned}
	c_{\eta, \tau}(h_1) \,\partial^\eta h_1 \,\partial^\tau h_1 &-
		c_{\eta,\tau}(h_2) \,\partial^\eta h_2 \,\partial^\tau h_2 = \\
		&c_{\eta,\tau}(h_1)\Big(\partial^\eta h_1(\partial^\tau h_1 - \partial^\tau h_2)
			+ \partial^\tau h_2(\partial^\eta h_1 - \partial^\eta h_2)\Big)\\
		&+\partial^\eta h_2 \, \partial^\tau h_2 \Big(c_{\eta,\tau}(h_1) - c_{\eta,\tau}(h_2)\Big)
\end{aligned}
\end{equation}
and note that $c_{\eta,\tau}$ is locally Lipschitz continuous in $E_\mu$ in order
to bound the previous expression in $E_0$ by
\[
	\tilde{C}\Big(|h_1|_{E_\beta}|h_1-h_2|_{bc^{2+\alpha}}
		+ |h_2|_{bc^{2+\alpha}}|h_1-h_2|_{bc^{3+\alpha}}
		+ |h_2|_{bc^{2+\alpha}}|h_2|_{E_\beta}|h_1-h_2|_{bc^{1+\alpha}}\Big).
\]
Further, by the reiteration theorem for continuous interpolation (c.f. \cite[Section~I.2.8]{Am95}), we have 
$bc^{2+\alpha} = (E_\mu, E_\beta)_{1/2,\infty}^{0}$. 
Thus,
\begin{equation}\label{reiteration}
	|h_i|_{bc^{2+\alpha}} \le \tilde{C}|h_i|_{bc^{3+\alpha}}^{1/2} 
		= \tilde{C}|h_i|_{E_\beta}^{1/2}
\end{equation}
and so, when $|\tau| = 2$, we bound
\[
\begin{aligned}
	|c_{\eta, \tau}(h_1) \, \partial^\eta h_1\, \partial^\tau h_1 &- 
		c_{\eta, \tau}(h_2) \, \partial^\eta h_2\, \partial^\tau h_2|_{E_0} \le
		\tilde{C}\Big(|h_1|_{E_\beta}|h_1-h_2|_{bc^{2+\alpha}}\\
		&+ |h_2|^{1/2}_{E_\beta}|h_1-h_2|_{bc^{3+\alpha}}
		+ |h_2|^{3/2}_{E_\beta}|h_1-h_2|_{bc^{1+\alpha}}\Big),
\end{aligned}
\]
which are three critical terms with $(\rho_j, \beta_j) = (1,1/2)$,
$(\rho_j, \beta_j) = (1/2, 3/4),$ and $(\rho_j, \beta_j)=(3/2,1/4),$ respectively. 

Moving on to the second term in \eqref{F2Structure}, we first rewrite the difference
\begin{equation}\label{dDifference}
\begin{aligned}
	d_{\eta,\tau,\sigma}&(h_1)\, \partial^\eta h_1 \, \partial^\tau h_1 \, \partial^\sigma h_1
	- d_{\eta,\tau,\sigma}(h_2)\, \partial^\eta h_2 \, \partial^\tau h_2 \, \partial^\sigma h_2\\
	&= d_{\eta,\tau,\sigma}(h_1)
		\Big(\partial^\eta h_1 \, \partial^\tau h_1(\partial^\sigma h_1 - \partial^\sigma h_2)\\
	&\qquad+ \partial^\eta h_1 \, \partial^\sigma h_2(\partial^\tau h_1 - \partial^\tau h_2)
	+ \partial^\tau h_2 \, \partial^\sigma h_2 (\partial^\eta h_1 - \partial^\eta h_2)\Big)\\
	&\qquad+ \partial^\eta h_2 \, \partial^\tau h_2 \, \partial^\sigma h_2
		\Big(d_{\eta,\tau,\sigma}(h_1)-d_{\eta,\tau,\sigma}(h_2)\Big)
\end{aligned}
\end{equation}
and then we proceed by splitting our analysis into three cases.

First, consider the case when only one of the terms $|\eta|, |\tau|,$ or $|\sigma|$ is 
equal to two. In this case, we again employ local Lipschitz continuity of the functions 
$d_{\eta,\tau,\sigma}$ and properties of interpolation to bound \eqref{dDifference}
in $E_0$ by
\[
	\tilde{C}\Big(
	|h_1|_{E_\beta}^{1/2}|h_1-h_2|_{bc^{1+\alpha}} + 
	|h_2|_{E_\beta}^{1/2}|h_1-h_2|_{bc^{1+\alpha}} + 
	|h_1 -h_2|_{bc^{2+\alpha}} +
	|h_2|_{E_\beta}^{1/2}|h_1-h_2|_{bc^{1+\alpha}}\Big),
\]
which results in two subcritical terms (after combining expressions)
with $(\rho_j, \beta_j) = (1/2,1/4)$ and $(\rho_j, \beta_j) = (0,1/2).$

Next, we consider the case when exactly two of the terms $|\eta|, |\tau|,$ and
$|\sigma|,$ equal two. Without loss of generality, we take $|\sigma|=1$, since
we can arrive at similar expressions to \eqref{dDifference} with any combination of multi--indices
contained in the cross--term $\partial^\eta h_1 \, \partial^\sigma h_2$.
Similar to the previous case, we majorize \eqref{dDifference} in $E_0$,
using \eqref{reiteration},
\[
\begin{aligned}
	&\tilde{C}\Big(
	|h_1|_{bc^{2+\alpha}}^2|h_1-h_2|_{bc^{1+\alpha}} +
	|h_1|_{bc^{2+\alpha}}|h_1-h_2|_{bc^{2+\alpha}} \\
	&\qquad +|h_2|_{bc^{2+\alpha}}|h_1-h_2|_{bc^{2+\alpha}} +
	|h_2|_{bc^{2+\alpha}}^2|h_1-h_2|_{bc^{1+\alpha}}\Big)\\
	&\le \tilde{C}\Big(
	|h_1|_{E_\beta}|h_1-h_2|_{bc^{1+\alpha}} +
	|h_1|_{E_\beta}^{1/2}|h_1-h_2|_{bc^{2+\alpha}} \\
	&\qquad\quad +|h_2|_{E_\beta}^{1/2}|h_1-h_2|_{bc^{2+\alpha}} +
	|h_2|_{E_\beta}|h_1-h_2|_{bc^{1+\alpha}}\Big),
\end{aligned}
\]
which contributes two additional subcritical terms with $(\rho_j, \beta_j) = (1,1/4)$
and $(\rho_j, \beta_j) = (1/2,1/2).$

Finally, in case $|\eta|=|\tau|=|\sigma|=2$, we bound \eqref{dDifference} by,
again employing \eqref{reiteration},
\[
\begin{aligned}
	&\tilde{C}\Big(
	\big(|h_1|_{E_\beta} + 
	|h_1|_{E_\beta}^{1/2}|h_2|_{E_\beta}^{1/2} + 
	|h_2|_{E_\beta} \big)|h_1-h_2|_{bc^{2+\alpha}} +
	|h_2|_{E_\beta}^{3/2}|h_1-h_2|_{bc^{1+\alpha}}\Big)\\
	&\le \tilde{C}\Big(
	|h_1|_{E_\beta}|h_1-h_2|_{bc^{2+\alpha}} +
	|h_2|_{E_\beta}|h_1-h_2|_{bc^{2+\alpha}} +
	|h_2|_{E_\beta}^{3/2}|h_1-h_2|_{bc^{1+\alpha}}\Big),
\end{aligned}
\]
where we have also applied Young's inequality to get 
\[
	|h_1|_{E_\beta}^{1/2}|h_2|_{E_\beta}^{1/2} \le 
		\frac{1}{2}|h_1|_{E_\beta} + \frac{1}{2}|h_2|_{E_\beta}.
\]
We thus produce two critical terms in this case, with
$(\rho_j, \beta_j) = (1,1/2)$ and $(\rho_j, \beta_j) = (3/2,1/4).$

Together with the analysis for the first term of \eqref{F2Structure}, we have
thus demonstrated that $F_2$ satisfies the structural condition {\bf (H2)}
and we can now apply Theorem~\ref{thm:MainResult} to produce the following results. 
Moreover, we note that $\mu = 1/4$ is the critical weight for $F_2,$ which 
indicates that $bc^{1+\alpha}(\mathcal{C}_r)$ is the critical space for \eqref{SDFlow}.

\medskip
\begin{thm}[{\bf Well--Posedness of \eqref{SDFlow}}] \label{thm:SDFlowWellPosed}
	Fix $\varepsilon > 0$ and $\alpha \in (0,1)$. 
\begin{itemize}
	\item[{\bf (a)}] For each initial value 
	\[
		h_0 \in V_\mu := bc^{1 + \alpha}(\mathcal{C}_r) \cap [h > \varepsilon - r],
	\]
	there exists a unique maximal solution to \eqref{SDFlow} with the addition property
	\[
		h(\cdot, h_0) \in C([0,t^+(h_0)), bc^{1 + \alpha}(\mathcal{C}_r)) 
			\cap C((0,t^+(h_0)), bc^{4 + \alpha}(\mathcal{C}_r)).
	\]
	Further, it follows that the map $[(t, h_0) \mapsto h(t,h_0)]$
	defines a semiflow on $V_\mu$ which is analytic for $t > 0$ and Lipschitz
	continuous for $t \ge 0$.
	
	\medskip
	\item[{\bf (b)}]
	Moreover, if the solution $h(\cdot, h_0)$ satisfies:
	\begin{itemize}
		\vspace{1mm}\item[(i)] $h(\cdot, h_0) \in UC(J(h_0),bc^{1+\alpha}(\mathcal{C}_r)),$ and
		\item[(ii)] there exists $M > 0$ so that, for all $t \in J(h_0) := [0,t^+(h_0))$,
		\begin{itemize}
			\vspace{1mm}\item[(ii.a)] $h(t,h_0)(p) \ge 1/M - r$ 
				for all $p \in \mathcal{C}_r,$ and
			\vspace{1mm}\item[(ii.b)] $| h(t,h_0) |_{bc^{1 + \alpha}(\mathcal{C}_r)} \le M$,
		\end{itemize}
	\end{itemize}
	then it holds that $t^+(h_0) = \infty$ and $h(\cdot, h_0)$ is a 
	global solution of \eqref{SDFlow}.
\end{itemize}
\end{thm}

\medskip
\begin{remark}
{\bf (a)} 
Of particular note in Theorem~\ref{thm:SDFlowWellPosed}, control of 
first--order derivatives of $h$ is sufficient to determine the lifespan of maximal solutions.
This improves previous results in the axially--definable setting (c.f. \cite[Proposition 3.2]{LS16})
which required control of second--order derivatives of solutions as well.

\medskip
{\bf (b)}
Regarding analyticity of the semiflow $[(t,h_0)\mapsto h(t,h_0)]$ for $t > 0$: for
any $\tau > 0$ we note that 
\[
	h(\tau, h_0) \in V_\beta := bc^{3+\alpha}(\mathcal{C}_r) \cap [h > \varepsilon - r].
\]
Thus, analyticity holds for $t > \tau$ in $V_\beta$ by \cite[Theorem 6.1]{CS01} and 
\eqref{A-F1-F2-regularity}, and then analyticity also holds in $V_\mu$ by embedding.

\medskip
{\bf (c)}
In the setting of surfaces expressed as graphs over $\R^n$, existence and uniqueness of
solutions with initial data in $bc^{1+\alpha}(\R^n)$ was established in \cite[Theorem 4.2]{Asai12}. 
However, we note that the author requires initial values to be slightly more regular than 
those considered here, due to the fact that he tracks regularity of solutions in a different 
topology than that of the space where he takes initial data. 
\end{remark}

\medskip
\subsection{Axially--Definable Setting: Stability of Cylinders}

Considering the stability of cylinders as equilibria for \eqref{SDFlow}, we first introduce 
the $2\pi$--periodic little--H\"older spaces $bc_{per}^{k+\alpha}, k \in \mathbb{N},$ 
defined as the subspace of functions $h \in bc^{k + \alpha}(\mathcal{C}_r)$ exhibiting
$2\pi$--periodicity along the $x$--axis; i.e.
\[
	h(x\pm2\pi, \theta) = h(x,\theta) \qquad \text{for all $(x,\theta) \in \mathcal{C}_r.$}
\]
As shown in \cite[Sections 3.4--4.1]{LS16}, working in this setting allows access 
to Fourier series representations for height functions $h$ and guarantees 
the linearized operator $DG(h)$ has a discrete spectrum. 

Regarding well--posedness in the periodic setting, it was shown in \cite[Proposition~3.4]{LS16}
that $G$ preserves periodicity, so Theorem~\ref{thm:SDFlowWellPosed}(a) continues to
hold verbatim with $bc_{per}$ replacing $bc$ throughout. Meanwhile, we note that
global solutions in the periodic setting differ slightly from Theorem~\ref{thm:SDFlowWellPosed}(b)
owing to the compactness of the embedding 
$bc_{per}^{4+\alpha}(\mathcal{C}_r) \hookrightarrow bc_{per}^{\alpha}(\mathcal{C}_r)$
(c.f. Theorem~\ref{thm:MainResult}(c)). 

Noting that $h_* \equiv 0$ is always an equilibrium of \eqref{SDFlow} (which coincides with the observation
that the cylinder $\mathcal{C}_r$ is an equilibrium of surface diffusion flow), we consider the stability
of $h_*$ under perturbations in 
\[
	V_{\mu, per} := bc_{per}^{1+\alpha}(\mathcal{C}_r) \cap [h > -r].
\]
Further, we denote by $\mathcal{M}_{cyl}$ the family of height functions 
$\bar h$ such that $\Gamma(\bar h)$ defines a cylinder 
$C(\bar y, \bar z, \bar r)$ --- symmetric about axis $(\cdot,\bar y, \bar z)$ in $\R^3$
with radius $\bar r > 0$ --- in a neighborhood of $\mathcal{C}_r$.
With these preparations, we state the following stability result.

\medskip
\begin{thm}[Global existence, stability / instability of cylinders]\label{thm:SDFlowStability}
Fix $\alpha \in (0,1)$.
\begin{itemize}
\item[{\bf (a)}] {\bf (Global Existence)} Let $h_0 \in V_{\mu, per}$ and suppose there exists
a constant $M > 0$ so that
\begin{itemize}
	\vspace{1mm}\item[(i)] $h(t,h_0)(p) \ge 1/M - r$ for all $t \in J(h_0) = [0,t^+(h_0))$, 
		$p \in \mathcal{C}_r$, and
	\vspace{1mm}\item[(ii)] $| h(t,h_0) |_{bc^{1+\delta}} \le M$ for all $t \in [\tau, t^+(h_0))$, 
		for some $\tau \in \dot{J}(h_0)$ and $\delta \in (\alpha,1)$,
\end{itemize}\vspace{1mm}
then $t^+(h_0) = \infty,$ so that $h(\cdot, h_0)$ is a global solution.
\vspace{1mm}
\item[{\bf (b)}] {\bf (Stability)} Fix $r > 1$ and $\bar \mu \in (0,1).$ 
	There exists a positive constant $\delta > 0$ 
	such that, given any admissible periodic perturbation
	\[
		h_0 \in V_{\mu, per} := bc_{per}^{1+\alpha}(\mathcal{C}_r) \cap [h > -r]
	\]
	with $| h_0 |_{bc^{1+\alpha}(\mathcal{C}_r)} < \delta$, the solution $h(\cdot, h_0)$ exists
	globally in time and converges to some $\bar h \in \mathcal{M}_{cyl}$ at an 
	exponential rate, in the topology of $E_{\bar\mu}.$ 
\vspace{1mm}
\item[{\bf (c)}] {\bf (Instability)} For $0<r<1$ the function $h_* \equiv 0$ is unstable in the 
	topology of $bc_{per}^{1+\alpha}(\mathcal{C}_r).$
\end{itemize}
\end{thm}

\begin{proof}
{\bf (a)}
This result follows from Theorem~\ref{thm:MainResult}(c), noting that 
$bc^{4+\alpha}_{per}(\mathcal{C}_r) \hookrightarrow bc^{\alpha}_{per}(\mathcal{C}_r)$
is a compact embedding in this periodic setting. 
Conditions (i)--(ii) guarantee the solution remains bounded away from 
the boundary $\partial V_{\mu, per}$. 

\medskip
{\bf (b)}
First note that restricting the domains of $(A,F_1,F_2)$ to periodic little--H\"older spaces 
will maintain the conditions {\bf (H1)--(H2)} and \eqref{A-F1-F2-regularity}, all confirmed in  
Section~\ref{subsec:AxiallyDefinable}.
From the proof of \cite[Theorem 4.3]{LS16} we know that $h_*$ is normally stable when
$r > 1$. 
The conclusion thus follows from Theorem~\ref{thm:normally-stable}.

\medskip
{\bf (c)}
It follows from \cite[Theorem 4.3(b)]{LS16} that $h_*$ is unstable in the topology of
$bc^{3 + \alpha}_{per}(\mathcal{C}_r).$ Instability in $bc^{1+\alpha}_{per}(\mathcal{C}_r)$ 
then follows from Proposition~\ref{pro:stability}.
\end{proof}

\medskip
\subsection{Axisymmetric Setting}

We turn now to consider \eqref{SDFlow} acting on the scale of axisymmetric
little--H\"older spaces $bc^{k + \alpha}_{sym}(\mathcal{C}_r), k \in \mathbb{N},$
defined as the subspace of functions $h \in bc^{k + \alpha}(\mathcal{C}_r)$ exhibiting symmetry
around the $x$--axis; i.e.
\[
	h(x,\theta_1)=h(x,\theta_2) \qquad \text{for all $x \in \R$ and $\theta_1, \theta_2 \in \mathbb{T}$.}
\]
These functions naturally coincide with surfaces $\Gamma(h)$ which are symmetric
about the central $x$--axis, as considered in \cite{LS13}; although we relax the setting
slightly by not enforcing axial--periodicity for our well--posedness result.

For all such functions with sufficient regularity, it follows that $\partial_\theta h \equiv 0$ and 
the application of the evolution operator $G$ to $h \in bc^{4+\alpha}_{sym}(\mathcal{C}_r)$ 
produces the simplified expression
\begin{equation}\label{ASD}
	G(h) = \frac{1}{(r+h)} \, \partial_x \left[ \frac{(r + h)}{\sqrt{1 + h_x^2}} \; 
		\partial_x \left(\frac{1}{(r+h) \sqrt{1 + h_x^2}} - 
		\frac{h_{xx}}{(1 + h_x^2)^{\frac{3}{2}}} \right) \right].
\end{equation}
In fact, a complete expansion of individual terms for the operator $G$ is provided in 
\cite[Equations (2.1)--(2.3)]{LS13} from which we deduce
\[
	A(h) = \frac{1}{(1 + h_x^2)^2} \; \partial_x^4 \qquad \qquad
	F_1(h) = \frac{h_x^2}{(r+h)^3 (1 + h_x^2)}
\]
and
\[
\begin{aligned}
	F_2(h) &= \frac{2 h_x (1+h_x^2)}{(r+h)(1+h_x^2)^3} \; h_{xxx} +
		\frac{-6 (r+h)h_x}{(r+h)(1+h_x^2)^3} \; h_{xx}h_{xxx}\\
		&\qquad +
		\frac{h_x^2 - 1}{(r+h)^2 (1+h_x^2)^2} \; h_{xx} + 
		\frac{6 h_x^2-1}{(r+h)(1+h_x^2)^3} \; h_{xx}^2 +
		\frac{3-15 h_x^2}{(1+h_x^2)^4} \; h_{xx}^3 \ ,
\end{aligned}
\]
where we can explicitly observe the structure of \eqref{F2Structure}.

To apply the results of Section~\ref{subsec:AxiallyDefinable} to the axisymmetric setting,
it suffices to note that the property of axisymmetry is preserved by \eqref{SDFlow}.
This claim is clear from a purely geometric perspective, since the evolution equation 
\eqref{SDFlow} is completely determined by the geometry of the surfaces $\Gamma(h(t)),$
and axisymmetry of the surface imparts the same symmetry onto the geometric structure.
However, one can also confirm preservation of axisymmetry analytically by confirming that
$G$ commutes with the azimuthal shift operators $T_\phi,$ for  $\phi \in \mathbb{T};$ 
defined by
\[
	T_\phi (h(x,\theta)) := h(x, \theta+\phi) \qquad \text{for $(x,\theta) \in \mathcal{C}_r.$}
\]
Indeed, we note that axisymmetry of $h$ can be characterized by the property that $T_\phi h = h$ 
on $\mathcal{C}_r$ for all $\phi \in \mathbb{T}.$ Then, by direct computation one
confirms that $T_\phi G(h) = G(T_\phi h)$.
From here we apply Theorems~\ref{thm:SDFlowWellPosed} and
\ref{thm:SDFlowStability}, along with an argument similar to the proof of \cite[Proposition 3.6]{LS16}
to produce the following extensions of \cite[Propositions 2.2 and 2.3]{LS13}.

\medskip
\begin{thm}[Well--Posedness] \label{thm:ASDWellPosed}
	Fix $\varepsilon > 0$ and $\alpha \in (0,1)$. 
\begin{itemize}
	\item[{\bf (a)}] For each admissible axisymmetric initial value 
	\[
		h_0 \in V_{\mu,sym} := bc^{1 + \alpha}_{sym}(\mathcal{C}_r) \cap [h > \varepsilon - r],
	\]
	there exists a unique maximal solution to \eqref{SDFlow} with the additional property
	\[
		h(\cdot, h_0) \in C([0,t^+(h_0)), bc^{1 + \alpha}_{sym}(\mathcal{C}_r)) 
			\cap C((0,t^+(h_0)), bc^{4 + \alpha}_{sym}(\mathcal{C}_r)).
	\]
	Further, the map $[(t, h_0) \mapsto h(t,h_0)]$
	defines a semiflow on $V_{\mu,sym}$ which is analytic for $t > 0$ and Lipschitz
	continuous for $t \ge 0$.
	
	\medskip
	\item[{\bf (b)}] Moreover, if the solution $h(\cdot, h_0)$ satisfies:
	\begin{itemize}
		\vspace{1mm}
		\item[(i)] $h(\cdot, h_0) \in UC(J(h_0),bc^{1+\alpha}_{sym}(\mathcal{C}_r)),$ and
		\item[(ii)] there exists $M > 0$ so that, for all $t \in J(h_0) := [0,t^+(h_0))$,
		\begin{itemize}
			\vspace{1mm}
			\item[(ii.a)] $h(t,h_0)(p) \ge 1/M - r$ 
				for all $p \in \mathcal{C}_r,$ and
			\item[(ii.b)] $| h(t,h_0) |_{bc^{1 + \alpha}(\mathcal{C}_r)} \le M$,
		\end{itemize}
	\end{itemize}
	then it holds that $t^+(h_0) = \infty$ and $h(\cdot, h_0)$ is a 
	global solution of \eqref{SDFlow}.
\end{itemize}
\end{thm}

With the additional assumption of periodicity along the $x$--axis, we likewise
define $bc_{symm,per}^{k+\alpha}(\mathcal{C}_r)$, $k \in \mathbb{N}$, as the subspace
of periodic functions $h \in bc^{k + \alpha}_{symm}(\mathcal{C}_r)$, with
$h(x \pm 2\pi, \theta) = h(x,\theta)$ for all $(x,\theta) \in \mathcal{C}_r$. 

\medskip
\begin{thm}[Global existence, stability and instability of cylinders]\label{thm:ASDFlowStability}
Fix $\alpha \in (0,1)$.
\begin{itemize}
\vspace{1mm}
\item[{\bf (a)}] {\bf (Global Existence)} Let $h_0 \in V_{\mu, symm, per}$ and suppose there exists
a constant $M > 0$ so that, for all $t \in J(h_0)$,
\begin{itemize}
	\vspace{1mm}\item[(i)] $h(t,h_0)(p) \ge 1/M - r$ for all $t \in J(h_0) = [0,t^+(h_0))$, 
		$p \in \mathcal{C}_r$, and
	\vspace{1mm}\item[(ii)] $| h(t,h_0) |_{bc^{1+\delta}} \le M$ for all $t \in [\tau, t^+(h_0))$, 
		for some $\tau \in \dot{J}(h_0)$ and $\delta \in (\alpha,1)$,
\end{itemize}\vspace{1mm}
then $t^+(h_0) = \infty,$ so that $h(\cdot, h_0)$ is a global solution.
\vspace{1mm}
\item[{\bf (b)}] {\bf (Stability)} Fix $r > 1$ and $\bar \mu \in (0,1).$ 
	There exists a positive constant $\delta > 0$ 
	such that, given any admissible periodic axisymmetric perturbation
	\[
		h_0 \in V_{\mu, sym, per} := bc_{sym, per}^{1+\alpha}(\mathcal{C}_r) \cap [h > -r]
	\]
	with $| h_0 |_{bc^{1+\alpha}(\mathcal{C}_r)} < \delta$, the solution $h(\cdot, h_0)$ exists
	globally in time and converges to some $\bar h \in \mathcal{M}_{cyl}$ at an 
	exponential rate, in the topology of $E_{\bar\mu}.$ 
\vspace{1mm}
\item[{\bf (c)}] {\bf (Instability)} For $0<r<1$ the function $h_* \equiv 0$ is unstable in the 
	topology of $bc_{per}^{1+\alpha}(\mathcal{C}_r).$
\end{itemize}
\end{thm}

\medskip
\subsection{Surfaces Near Compact Hypersurfaces}

We conclude the paper by looking at the flow of surfaces parameterized 
over a fixed reference manifold and extend results in \cite{EMS98}. 
In particular, let $\Sigma$ denote a smooth, closed, 
compact, immersed, oriented hypersurface 
in $\R^n$, and let $\nu_\Sigma$ be a unit normal vector field on $\Sigma$, 
compatible with the chosen orientation. It follows that 
there exists a constant $a > 0$ and an open atlas 
$\{ U_\ell :\ell \in \mathcal{L} \}$ for $\Sigma$ so that
\begin{equation}\label{TubeNHood}
	X_\ell : U_\ell \times (-a,a) \to \R^n, \qquad 
		X_\ell(p,r) := p + r \nu_\Sigma(p),
\end{equation}
is a smooth diffeomorphism onto the range $R_\ell := im(X_\ell),$ for $\ell \in \mathcal{L}.$
We capture the evolution of surfaces that are $C^1$--close to $\Sigma$ via
time--dependent height functions $h : \R^+ \times \Sigma \to (-a,a).$
In particular, to $h(t) := h(t, \cdot)$ we associate the surface
\[
	\Gamma(h(t)) := \bigcup_{\ell \in \mathcal{L}} \{ X_\ell (p, h(t,p)) : p \in U_\ell \},
\]
which is parametrized by the mapping
\begin{equation}\label{Psi}
	\Psi_{h(t)} : \Sigma \to \R^n, \qquad
		\Psi_{h(t)}(p) := p + h(t,p) \nu_\Sigma(p).
\end{equation}

As in the previous settings, we let $\alpha \in (0,1)$ and work in spaces of little--H\"older
continuous functions
\[
	E_0 := bc^{\alpha}(\Sigma), \quad E_\mu = bc^{1+\alpha}(\Sigma), \quad
	E_\beta = bc^{3+\alpha}(\Sigma), \quad E_1 := bc^{4 + \alpha}(\Sigma),
\]
with $\mu = 1/4$ and $\beta = 3/4$. 
Further, we define
\[
	V_\mu := bc^{1+\alpha}(\Sigma) \cap [|h|_{C(\Sigma)} < a].
\]

To express the equations for surface diffusion flow of $\Gamma(h(t))$ as
an evolution equation acting on the height functions $h$, we direct the 
reader to \cite[Section 2]{EMS98} and \cite[Section 5]{S15} where details are 
given for pulling back the governing equation
$V_{\Gamma} = \Delta_{\Gamma} \mathcal{H}_{\Gamma},$
defined on $\Gamma(h(t)),$ to an equivalent equation on the reference manifold $\Sigma$.
We thus arrive at an expression
\begin{equation}\label{SDFlowOnSigma}
\begin{cases}
	h_t(t,p) = [G(h(t))](p) &\text{for $t > 0, \; p \in \Sigma$},\\
	h(0) = h_0 &\text{on $\Sigma$},
\end{cases}
\end{equation}
where the evolution operator $G$ takes the form (c.f. \cite[Section 5]{S15})
\begin{equation}\label{GOnSigma}
	G(h) := -\frac{1}{\beta(h,\partial^1 h)} \Delta_h \mathcal{H}_h.
\end{equation} 

Utilizing expressions given in \cite[Sections 3.2--3.5]{PS13} or  \cite[Section 2.2]{PS16}, 
we expand \eqref{GOnSigma} and confirm properties {\bf (H1), (H2)} and \eqref{A-F1-F2}.
For the structure of the Laplace--Beltrami operator in local coordinates,
we have (employing the standard summation convention over repeated instances 
of $i, j, k$ taking values from 1 to $(n-1)$ --- the dimension of the manifold)
\begin{equation}\label{LaplaceStructure}
	\Delta_h \varphi = a^{ij}(h,\partial^1 h) \partial_i \partial_j \varphi_* 
		+ b^k(h, \partial^1 h, \partial^2 h) \partial_k \varphi_*,
\end{equation}
where $\varphi$ is a scalar function on $\Gamma(h)$ and $\varphi_* := \Psi_h^* \varphi$
its pull--back to $\Sigma$ through the parameterization $\Psi_h$.
The coefficient functions $a^{ij}$ and $b^k$ are expressed as
\[
\begin{aligned}
	a^{ij}(h, \partial^1 h) &= \Big( P_{\Gamma(h)} M_0(h) \tau_\Sigma^i \Big|
		P_{\Gamma(h)} M_0(h) \tau_\Sigma^j \Big) \qquad \text{and}\\
	b^k(h, \partial^1 h, \partial^2 h) &= \Big( \partial_i \big( M_0(h) P_{\Gamma(h)}\big)
		P_{\Gamma(h)} M_0(h) \tau_\Sigma^i \Big| \tau_\Sigma^k \Big),		
\end{aligned}
\]
where $\tau^\Sigma_i \big|_p$ denote elements of a basis for the tangent space $T_p \Sigma$ 
to a point $p \in \Sigma$, while $\tau_\Sigma^i \big|_p$ make up a corresponding basis for the 
dual of $T_p \Sigma,$ and $\big( \cdot \big| \cdot \big)$ is the inner product in Euclidean space $\R^n$.
Further, $M_0(h) := (I - h L_\Sigma)^{-1}$ depends upon $h$ and the Weingarten tensor $L_\Sigma$ 
on $\Sigma$ (i.e. no derivatives of $h$ appear in $M_0(h)$) 
and 
\[
	P_{\Gamma(h)} :=  I - \nu_{\Gamma(h)} \otimes \nu_{\Gamma(h)}
\]
projects onto the tangent space $\Gamma(h)$. Here the normal vector to $\Gamma(h)$ is
\[
	\nu_{\Gamma(h)} = \beta(h,\partial^1 h) ( \nu_{\Sigma} - M_0(h) \nabla_\Sigma h),
\]
and hence $P_{\Gamma(h)}$ only depends upon first--order derivatives of $h$.
Therefore, second--order derivatives of $h$ only appear in $b^k$ when 
the derivative $\partial_i$ acts on $P_{\Gamma(h)}$ in the first term of the inner product. 
With no further factors of $\partial^2 h$ appearing in the expression, it follows that 
$\partial^2 h$ only appears linearly in the functions $b^k(h, \partial^1 h, \partial^2 h).$ 

By \cite[Section 2.2]{PS16} or \cite[Section 3.5]{PS13}, the mean curvature function has the 
following structure in local coordinates,
\begin{equation}\label{CurvatureStructure}
	\mathcal{H}_h = c^{ij}(h, \partial^1 h) \partial_i \partial_j h + d(h, \partial^1 h).
\end{equation}
Thus, applying \eqref{LaplaceStructure} to \eqref{CurvatureStructure}, one confirms that $G$
exhibits the quasilinear structure 
\[
	G(h) = -A(h, \partial^1 h) h + F_1(h, \partial^1 h) +
		F_2(h, \partial^1 h, \partial^2 h, \partial^3 h).
\]
Considering condition {\bf (H1)}, note that, in every local chart $U_\ell$, the principal symbol
$\sigma[A(h)]_\ell$ coincides with the expression given for the principal symbol 
$\hat{\sigma}[P(h)]_\ell$ in \cite[Section 5]{S15}. Thus, we have
\[
	A(h) \in \mathcal{M}_\mu(E_1,E_0) \qquad \text{for all $h \in V_\mu,$}
\]
and we likewise conclude
\begin{equation}\label{RegularityOnSigma}
\begin{aligned}
	(A,F_1) &\in C^{\omega} \big( V_\mu, \mathcal{M}_\mu(E_1, E_0) \times E_0 \big)\\
	F_2 &\in C^{\omega} \big( V_\mu \cap E_\beta, E_0 \big).
\end{aligned}
\end{equation}

Considering condition {\bf (H2)}, by the argument in Section~\ref{subsec:AxiallyDefinable},
it suffices to show that $F_2$ exhibits the same structure as \eqref{F2Structure}. 
Thus, applying \eqref{LaplaceStructure} to \eqref{CurvatureStructure} (noting
that \eqref{CurvatureStructure} is the pulled back expression of mean curvature), we  
first consider the four scenarios where third--order derivatives arise in $G(h)$, namely:
\begin{itemize}
	\item $a^{ij}(h,\partial^1 h) \, \partial_i c^{kl}(h, \partial^1 h) \, \partial_j \partial_k \partial_l h$,
	\vspace{1mm}\item $a^{ij}(h,\partial^1 h) \, \partial_i \partial_j c^{kl}(h, \partial^1 h)  \, \partial_k \partial_l h$,
	\vspace{1mm}\item $b^k(h, \partial^1 h, \partial^2 h) \, c^{ij}(h, \partial^1 h) \, \partial_k \partial_i \partial_j h$, and
	\vspace{1mm}\item $a^{ij}(h,\partial^1 h) \, \partial_i \partial_j d(h, \partial^1 h)$.
\end{itemize}
In all such scenarios, when $\partial^3 h$ is produced it appears linearly and it multiplies factors of 
$\partial^2$ that appear at most linearly. 
Next, considering all cases within which second--order derivatives arise in $G(h)$ --- without accompanying
factors of $\partial^3 h$ --- one likewise confirms that at most cubic factors 
of $\partial^2 h$ appear. Therefore, we conclude that conditions {\bf (H1), (H2)} and \eqref{A-F1-F2}  
all hold for \eqref{SDFlowOnSigma}, and we thus produce the following extension of 
\cite[Theorem 2.2]{EMS98} by application of Theorem~\ref{thm:MainResult}. 
Note that the embeddings $bc^{k+\alpha}(\Sigma) \hookrightarrow bc^{\alpha}(\Sigma)$
are compact here, since the domain $\Sigma$ is itself compact.

\begin{thm}[Well--Posedness] \label{thm:SDOnSigmaWellPosed}
	Fix $\alpha \in (0,1)$ and $a > 0$. Let $\Sigma$ be a smooth, closed, compact, immersed, oriented
	hypersurface in $\R^n$ on which there exists an open atlas $\{ U_\ell: \ell \in \mathcal{L}\}$
	where $X_\ell: \big[(p,r) \mapsto p + r \nu_\Sigma(p)\big]: \Sigma \times (-a,a) \to \R^n$ is
	a smooth diffeomorphism onto $R_\ell := im(X_\ell)$, for $\ell \in \mathcal{L}$. 
\begin{itemize}
	\item[{\bf (a)}] For each admissible initial value 
	\[
		h_0 \in V_{\mu} := bc^{1 + \alpha}(\Sigma) \cap [ |h|_{C(\Sigma)} < a],
	\]
	there exists a unique maximal solution to \eqref{SDFlowOnSigma} with the additional property
	\[
		h(\cdot, h_0) \in C([0,t^+(h_0)), bc^{1 + \alpha}(\Sigma)) 
			\cap C((0,t^+(h_0)), bc^{4 + \alpha}(\Sigma)).
	\]
	Further, the map $[(t, h_0) \mapsto h(t,h_0)]$
	defines a semiflow on $V_{\mu}$ which is analytic for $t > 0$ and Lipschitz
	continuous for $t \ge 0$.
	
	\medskip
	\item[{\bf (b)}] Moreover, if there exists $M > 0$ and $\delta \in (\alpha, 1)$ so that
		\begin{itemize}
			\vspace{1mm}
			\item[(i)] $|h(t,h_0)(p)| \le a - 1/M,$ 
				for all $p \in \Sigma$ and $t \in J(h_0)$, and
			\vspace{1mm}\item[(ii)] $| h(t,h_0) |_{bc^{1 + \delta}(\Sigma)} \le M$,
				for all $t\in [\tau, t^+(h_0)),$ and some $\tau \in (0, t^+(h_0))$,
				\vspace{1mm}
		\end{itemize}
	then it holds that $t^+(h_0) = \infty$ and $h(\cdot, h_0)$ is a 
	global solution of \eqref{SDFlowOnSigma}.
\end{itemize}
\end{thm}

\begin{remark}
	{\bf (a)}
	We note that the global existence result in Theorem~\ref{thm:SDOnSigmaWellPosed}(b) is
	limited, as it fails to account for the possibility of updating the reference manifold $\Sigma$ 
	as $\Gamma(h(t))$ is leaving the tubular neighborhood, but this result is sufficient for 
	considerations of stability/instability when $\Sigma$ is an equilibrium.
	
	\medskip
	{\bf (b)}
	Further regularity of the surfaces $\Gamma(h(t))$ have been shown in certain settings.
	In particular, when $\Sigma$ is additionally assumed to be a smooth embedded surface, 
	it follows from \cite[Theorem 5.2]{S15}, and instantaneous regularization of solutions, 
	that $\Gamma(h(t))$ is also smooth for all $t \in \dot{J}(h_0)$. 
	Likewise, if $\Sigma$ is real analytic and embedded, then $\Gamma(h(t))$
	is also real analytic.
\end{remark}

\medskip
\subsection{Stability of Euclidean Spheres}

In the particular case that $\Sigma$ is a Euclidean sphere, the function $h_* \equiv 0$ is 
normally stable in $E_1$ by \cite[Section 3]{EMS98} and we thus conclude the following
extension of \cite[Theorem 1.2]{EMS98}. Note that our result shows stability of spheres 
under surface diffusion flow with control on only first derivatives of perturbations.

\begin{thm}\label{thm:Spheres}
	Fix $\alpha \in (0,1),$ $\bar\mu \in (0, 1).$ Let $\Sigma$ be a Euclidean sphere in 
	$\R^n$ and choose $a > 0$ so that the mapping $[(p,r) \mapsto p + r\nu_{\Sigma}(p)]$ is
	a diffeomorphism on $\Sigma \times (-a,a)$. 
	There exists a constant $\delta > 0$ such that, given any admissible
	perturbation $\Gamma(h_0)$ for
	$h_0 \in V_\mu$
	with $|h_0|_{bc^{1+\alpha}(\Sigma)} < \delta$, the solution $h(\cdot,h_0)$ exists
	globally in time and converges to some $\bar h \in \mathcal{M}_{sph}$ at an exponential
	rate, in the topology of $E_{\bar \mu}.$ Here $\mathcal{M}_{sph}$ denotes the family
	of all spheres which are sufficiently close to $\Sigma$ in $\R^n.$
\end{thm}

An immediate consequence of Theorem~\ref{thm:Spheres} is a relaxation of convexity constraints
for stable perturbations of a sphere. In particular, note that every $bc^{1+\alpha}$--neighborhood 
of a sphere contains non--convex hypersurfaces.
This corollary provides a different approach to the same result in \cite{EM10}, 
where the authors prove the claim by showing that non--convex perturbations 
of spheres exist in $B_{p,2}^{5/2-4/p}(\Sigma)$.

\begin{cor}
	There exist non--convex hypersurfaces $\Gamma_0$ such that the solution $h(\cdot,h_0)$ to
	\eqref{SDFlowOnSigma}, with $\Gamma(h_0) = \Gamma_0,$ exists globally in time and 
	converges exponentially fast to a sphere.
\end{cor}

\bigskip\bigskip
\bibliographystyle{plain}
\bibliography{Bib}

\end{document}